\documentclass[a4paper,dvipdmx]{amsart}

\usepackage{amsmath,amssymb,amsthm}
\usepackage{eucal}
\usepackage{tikz}
\usepackage[colorlinks=true, backref=page]{hyperref}
\usepackage[all]{xy}
\usepackage{here}

\numberwithin{equation}{section}

\SelectTips{eu}{12}

\newtheorem{theorem}{Theorem}[section]
\newtheorem{proposition}[theorem]{Proposition}
\newtheorem{lemma}[theorem]{Lemma}
\newtheorem{corollary}[theorem]{Corollary}

\theoremstyle{definition}

\newtheorem{example}[theorem]{Example}
\newtheorem{conjecture}[theorem]{Conjecture}

\theoremstyle{remark}
\newtheorem{remark}[theorem]{Remark}

\newcommand{\Z}{\mathbb{Z}}
\newcommand{\Q}{\mathbb{Q}}
\newcommand{\R}{\mathbb{R}}
\newcommand{\C}{\mathbb{C}}

\renewcommand{\t}{\mathfrak{t}}

\newcommand{\Hom}{\operatorname{Hom}}
\newcommand{\Rep}{\operatorname{Rep}}
\newcommand{\hocolim}{\operatorname{hocolim}}
\newcommand{\rank}{\operatorname{rank}}
\newcommand{\Int}{\operatorname{Int}}

\title[Torsion in the space of commuting elements in a Lie group]{Torsion in the space of commuting elements in a Lie group}

\author{Daisuke Kishimoto}
\address{Faculty of Mathematics, Kyushu University, Fukuoka, 819-0395, Japan}
\email{kishimoto@math.kyushu-u.ac.jp}

\author{Masahiro Takeda}
\address{Department of Mathematics, Kyoto University, Kyoto, 606-8502, Japan}
\email{takeda.masahiro.87u@st.kyoto-u.ac.jp}



\begin{document}

  \maketitle

  \begin{abstract}
    Let $G$ be a compact connected Lie group, and let $\Hom(\Z^m,G)$ be the space of pairwise commuting $m$-tuples in $G$. We study the problem of which primes $p$ $\Hom(\Z^m,G)_1$, the connected component of $\Hom(\Z^m,G)$ containing the element $(1,\ldots,1)$, has $p$-torsion in homology. We will prove that $\Hom(\Z^m,G)_1$ for $m\ge 2$ has $p$-torsion in homology if and only if $p$ divides the order of the Weyl group of $G$ for $G=SU(n)$ and some exceptional groups. We will also compute the top homology of $\Hom(\Z^m,G)_1$ and show that $\Hom(\Z^m,G)_1$ always has 2-torsion in homology whenever $G$ is simply-connected and simple. Our computation is based on a new homotopy decomposition of $\Hom(\Z^m,G)_1$, which is of independent interest and enables us to connect torsion in homology to the combinatorics of the Weyl group.
  \end{abstract}


  \section{Introduction}\label{Introduction}

  Let $\pi$ be a discrete group, and let $G$ be a compact connected Lie group. Let $\Hom(\pi,G)$ denote the space of homomorphisms from $\pi$ to $G$, having the induced topology of the space of continuous maps from $\pi$ to $G$. In this paper, we study torsion in the homology of $\Hom(\Z^m,G)_1$, the connected component of $\Hom(\Z^m,G)$ containing the trivial homomorphism.

  The space $\Hom(\pi,G)$ has connections to diverse contexts of mathematics and physics \cite{BFM,G,KS,We,W1,W2}, and the topology of $\Hom(\pi,G)$ has been intensely studied in recent years, especially when $\pi$ is a free abelian group. The space $\Hom(\Z^m,G)$ is identified with the space of commuting $m$-tuples in $G$, so that it is often called the \emph{space of commuting elements}. See \cite{AC,ACG,AGG,B,BS,BJS,C,GPS,KT,RS1,RS2,T} and references therein. In particular, Baird \cite{B} described the cohomology of $\Hom(\Z^m,G)_1$ over a field of characteristic not dividing the order of the Weyl group of $G$ or zero as a certain ring of invariants of the Weyl group. Based on this result, Ramras and Stafa \cite{RS1} gave a formula for the Poincar\'e series of $\Hom(\Z^m,G)_1$. We start with recalling this formula. Let $W$ denote the Weyl group of $G$, and let $\mathbb{F}$ be a field of characteristic not dividing the order of $W$ or zero. Then Ramras and Stafa \cite{RS1} proved that the Poincar\'e series of the cohomology of $\Hom(\Z^m,G)_1$ over $\mathbb{F}$ is given by
  \begin{equation*}
    \label{Poincare series}
    \frac{\prod_{i=1}^r(1-t^{2d_i})}{|W|}\sum_{w\in W}\frac{\det(1+tw)^m}{\det(1-t^2w)},
  \end{equation*}
  where $d_1,\ldots,d_r$ are the characteristic degrees of $W$, that is, the rational cohomology of $G$ is an exterior algebra generated by elements of degrees $2d_1-1,\ldots,2d_r-1$. A more explicit formula for the Poincar\'e series was obtained by the authors \cite{KT}, and a minimal generating set of the cohomology over $\mathbb{F}$ was also obtained there. An explicit description of the cohomology of $\Hom(\Z^2,G)_1$ over $\mathbb{F}$ for $G$ of rank two was obtained by the second author \cite{T}. Notice that the Poincar\'e series is independent of the ground field $\mathbb{F}$ as long as its characteristic does not divide the order of $W$ or is zero. Then we immediately get the non-existence of torsion in homology.

  \begin{lemma}
    \label{possible torsion}
    The homology of $\Hom(\Z^m,G)_1$ has $p$-torsion in homology only when $p$ divides the order of $W$.
  \end{lemma}

  On the other hand, as for the existence of torsion in the homology of $\Hom(\Z^m,G)_1$, there are only a few results, the proofs of which do not extend to more general cases. Adem and Cohen \cite{AC} proved a stable splitting of $\Hom(\Z^m,G)$, and Baird, Jeffrey and Selick \cite{BJS} and Crabb \cite{C} described the splitting summands for $G=SU(2)$ explicitly. As a result, we can conclude that $\Hom(\Z^m,SU(2))_1$ has 2-torsion in homology for $m\ge 2$. Recently, Adem, G\'{o}mez and Gritschacher \cite{AGG} computed the second homology group of $\Hom(\Z^m,G)_1$, and so by combining with the result on the fundamental group by G\'{o}mez, Pettet and Souto \cite{GPS}, $\Hom(\Z^m,Sp(n))_1$ has 2-torsion in homology for $m\ge 3$. These are all known existence of torsion in homology so far.

  \subsection{Results}

  By Lemma \ref{possible torsion}, we must know the order of the Weyl group of a Lie group. Then we give a table of the order of the Weyl groups of compact simply-connected simple Lie groups.

  \renewcommand{\arraystretch}{1.3}

  \begin{table}[H]
    \centering
    \caption{}
    \label{order}
    \begin{tabular}{lllll}
      \hline
      Type&Lie group&Rank&$W$&$|W|$\\\hline
      $A_n$&$SU(n+1)$&$n$&$\Sigma_{n+1}$&$(n+1)!$\\
      $B_n$&$Spin(2n+1)$&$n$&$B_n$&$2^nn!$\\
      $C_n$&$Sp(n)$&$n$&$B_n$&$2^nn!$\\
      $D_n$&$Spin(2n)$&$n$&$B_n^+$&$2^{n-1}n!$\\
      $G_2$&$G_2$&$2$&-----&$2^2\cdot 3$\\
      $F_4$&$F_4$&$4$&-----&$2^7\cdot 3^2$\\
      $E_6$&$E_6$&$6$&-----&$2^7\cdot 3^4\cdot 5$\\
      $E_7$&$E_7$&$7$&-----&$2^{10}\cdot 3^4\cdot 5\cdot 7$\\
      $E_8$&$E_8$&$8$&-----&$2^{14}\cdot 3^5\cdot 5^2\cdot 7$\\\hline
    \end{tabular}
  \end{table}

  Now we state our results.

  \begin{theorem}
    \label{main torsion SU(n)}
    The homology of $\Hom(\Z^m,SU(n))_1$ for $m\ge 2$ has $p$-torsion if and only if $p\le n$.
  \end{theorem}

  Since the Weyl group of $SU(n)$ is of order $n!$, it follows from Lemma \ref{possible torsion} that $\Hom(\Z^m,SU(n))_1$ has $p$-torsion in homology for all possible primes $p$. Then we can say that the space $\Hom(\Z^m,G)_1$ is quite complicated. We will also prove a similar result for some exceptional groups.

  \begin{theorem}
    \label{main torsion exceptional}
    Let $G=G_2,F_4,E_6$. Then $\Hom(\Z^m,G)_1$ for $m\ge 2$ has $p$-torsion in homology if and only if $p$ does not divide the order of the Weyl group of $G$.
  \end{theorem}

  Then for $G=G_2,F_4,E_6$, $\Hom(\Z^m,G)_1$ for $m\ge 2$ has all possible torsion in homology, so that it is quite complicated too. We will also show the existence of some torsion in the homology of $\Hom(\Z^m,G)_1$ for other Lie groups $G$, though incomplete. See Section \ref{Computation of torsion in homology} and Corollary \ref{main 2-torsion} below for details. Our next result is on the top homology of $\Hom(\Z^m,G)_1$. See \cite{KT} for the top rational homology.

  \begin{theorem}
    \label{main top}
    Let $G$ be a compact simply-connected simple Lie group of rank $n$, and let
    \[
      d=
      \begin{cases}
        \dim G+n(m-1)-1&m\text{ is even}\\
        \dim G+n(m-1)&m\text{ is odd}.
      \end{cases}
    \]
    Then the top homology of $\Hom(\Z^m,G)_1$ is given by
    \[
      H_d(\Hom(\Z^m,G)_1)\cong
      \begin{cases}
        \Z/2&m\text{ is even}\\
        \Z&m\text{ is odd}.
      \end{cases}
    \]
  \end{theorem}

  Since $\Hom(\Z^2,G)_1$ is a retract of $\Hom(\Z^m,G)_1$ for $m\ge 2$, we immediately obtain the following corollary.

  \begin{corollary}
    \label{main 2-torsion}
    Let $G$ be a compact simply-connected simple Lie group. Then $\Hom(\Z^m,G)_1$ for $m\ge 2$ has 2-torsion in homology.
  \end{corollary}

  Let $\pi$ be a finitely generated nilpotent group whose abelianization is of rank $m$. We can extend our results to $\Hom(\pi,G)_1$ as follows. Let $G(\C)$ be a complexification of $G$. Then Bergeron \cite{B} proved that $\Hom(\pi,G)$ is a deformation retract of $\Hom(\pi,G(\C))$. Moreover Bergeron and Silberman \cite{BS} proved that there is a homotopy equivalence $\Hom(\pi,G(\C))_1\simeq\Hom(\Z^m,G(\C))_1$. Then we get a homotopy equivalence
  \[
    \Hom(\pi,G)_1\simeq\Hom(\Z^m,G)_1
  \]
  and so all the results above also hold for $\Hom(\pi,G)_1$.

  \subsection{Summary of computation}

  We compute the homology of $\Hom(\Z^m,G)_1$ in three steps; the first step is to give a new homotopy decomposition of $\Hom(\Z^m,G)_1$, namely, we will describe $\Hom(\Z^m,G)_1$ as a homotopy colimit, the second step is to extract the top line of (a variant of) the Bousfield-Kan spectral sequence for the homotopy colimit in the first step, and the third step is to encode the information of the top line extracted in the second step into the combinatorial data of the extended Dynkin diagram of $G$.

  Let $G$ act on $\Hom(\pi,G)$ by conjugation. Then the quotient space, denoted by $\Rep(\pi,G)$, is called the \emph{representation space} or the \emph{character variety}, which has been studied in a variety of contexts \cite{AB,CS,Hi}. We will show that if $G$ is simply-connected and simple, then $\Rep(\Z,G)$, the quotient of $\Hom(\Z,G)_1$, is naturally identified with the Weyl alcove which is an $n$-simplex whose facets are defined by simple roots and the highest root, where $G$ is of rank $n$. We consider the composite
  \[
    \pi\colon\Hom(\Z^m,G)_1\to\Hom(\Z,G)\to\Rep(\Z,G)=\Delta^n,
  \]
  where the first map is the $m$-th projection and the second map is the quotient map. We will see that the fiber of $\pi$ is constant as long as the point belongs to the interior of some face of $\Delta^n$. Let $\sigma_0$ denote the barycenter of a face $\sigma$ of $\Delta^n$, and let $P(\Delta^n)$ denote the face poset of $\Delta^n$. Then we get a functor
  \[
    F_m\colon P(\Delta^n)\to\mathbf{Top},\quad\sigma\mapsto\pi^{-1}(\sigma_0),
  \]
  and a new homotopy decomposition in the first step is:

  \begin{theorem}
    [Theorem {\ref{hocolim}}]
    \label{homotopy decomposition}
    Let $G$ be a compact simply-connected simple Lie group. Then there is a homeomorphism
    \[
      \Hom(\Z^m,G)_1\cong\hocolim F_m.
    \]
  \end{theorem}

  This homotopy decomposition seems to be of independent interest. We will see that if $m$ even, then for each $\sigma\in P(\Delta^n)$, $F_m(\sigma)$ is of dimension $\dim G+n(m-2)$ and its top homology is isomorphic with $\Z$. Thus we can consider the pinch map onto the top cell of $F_m(\sigma)$, which enables us to extract (a variant of) the Bousfield-Kan spectral sequence for $\hocolim F_m$, the second step. This pinch map onto the top cell can be explicitly described in terms faces of $\Delta^n$. Then since faces of $\Delta^n$ are defined by simple roots and the highest weight, the computation of the top line can be connected to the extended Dynkin diagram, which is the third step.

  \subsection*{Acknowledgement}

  The first author was supported by JSPS KAKENHI No. JP17K05248 and JP19K03473, and the second author was supported by JSPS KAKENHI No. JP21J10117.


  \section{Triangulation of a maximal torus}\label{Triangulation of a maximal torus}

  Hereafter, let $G$ denote a compact simply-connected simple Lie group such that $\rank G=n$ and $\dim G=d$. Let $T$ and $W$ denote a maximal torus and the Weyl group of $G$, respectively. This section constructs a $W$-equivariant triangulation of a maximal torus, which will play the fundamental role in our study. Let $\t$ be the Lie algebra of $G$, and let $\Phi$ be the set of roots of $G$. Recall that the Stiefel diagram is defined by
  \[
    \bigcup_{\substack{\alpha\in\Phi\\i\in\Z}}\alpha^{-1}(i)
  \]
  which is a subspace of $\t$, where each $\alpha^{-1}(i)$ is called a wall in the Stiefel diagram. For example, the Stiefel diagram of $Sp(2)$ is given as follows, where integer points are indicated by white points.

  
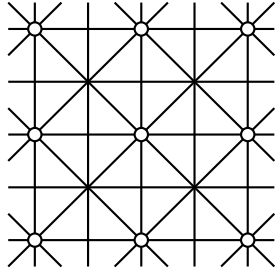
\begin{figure}[htbp]
    \centering
    \begin{tikzpicture}[x=0.7cm, y=0.7cm, thick]
      \draw(0.5,0.5)--(5.5,5.5);
      \draw(0.5,2.5)--(3.5,5.5);
      \draw(0.5,4.5)--(1.5,5.5);
      \draw(2.5,0.5)--(5.5,3.5);
      \draw(4.5,0.5)--(5.5,1.5);
      \draw(1.5,0.5)--(0.5,1.5);
      \draw(3.5,0.5)--(0.5,3.5);
      \draw(5.5,0.5)--(0.5,5.5);
      \draw(5.5,2.5)--(2.5,5.5);
      \draw(5.5,4.5)--(4.5,5.5);
      \draw(1,0.5)--(1,5.5);
      \draw(2,0.5)--(2,5.5);
      \draw(3,0.5)--(3,5.5);
      \draw(4,0.5)--(4,5.5);
      \draw(5,0.5)--(5,5.5);
      \draw(0.5,1)--(5.5,1);
      \draw(0.5,2)--(5.5,2);
      \draw(0.5,3)--(5.5,3);
      \draw(0.5,4)--(5.5,4);
      \draw(0.5,5)--(5.5,5);
      \fill[white](1,1) circle[radius=2.5pt];
      \draw(1,1) circle[radius=2.5pt];
      \fill[white](1,3) circle[radius=2.5pt];
      \draw(1,3) circle[radius=2.5pt];
      \fill[white](1,5) circle[radius=2.5pt];
      \draw(1,5) circle[radius=2.5pt];
      \fill[white](3,1) circle[radius=2.5pt];
      \draw(3,1) circle[radius=2.5pt];
      \fill[white](3,3) circle[radius=2.5pt];
      \draw(3,3) circle[radius=2.5pt];
      \fill[white](3,5) circle[radius=2.5pt];
      \draw(3,5) circle[radius=2.5pt];
      \fill[white](5,1) circle[radius=2.5pt];
      \draw(5,1) circle[radius=2.5pt];
      \fill[white](5,3) circle[radius=2.5pt];
      \draw(5,3) circle[radius=2.5pt];
      \fill[white](5,5) circle[radius=2.5pt];
      \draw(5,5) circle[radius=2.5pt];
    \end{tikzpicture}
  \end{figure}

  Since $G$ is simple, its Stiefel diagram is a simplicial complex such that every $k$-face is included in an intersection of exactly $n-k$ walls.

  \begin{lemma}
    \label{Stiefel diagram}
    If two vertices $v$ and $v+w$ of the Stiefel diagram of $G$ are joined by an edge, then $w$ is a vertex of the Stiefel diagram which is joined with the vertex $0$ by an edge.
  \end{lemma}

  \begin{proof}
    Since $v$ and $v+w$ are joined by an edge,
    \[
      \{v\}=\theta_1^{-1}(k_1)\cap\cdots\cap\theta_n^{-1}(k_n)\quad\text{and}\quad\{v+w\}=\theta_1^{-1}(k_1)\cap\cdots\cap\theta_{n-1}^{-1}(k_{n-1})\cap\theta_n^{-1}(k_n+\epsilon)
    \]
    for some roots $\theta_1,\ldots,\theta_n$ and integers $k_1,\ldots,k_n$, where $\epsilon=\pm 1$. Then
    \[
      \{w\}=\theta_1^{-1}(0)\cap\cdots\cap\theta_{n-1}^{-1}(0)\cap\theta_n^{-1}(\epsilon)
    \]
    which is a vertex of the Stiefel diagram. Moreover, since $\epsilon=\pm 1$, $0$ and $w$ are joined by an edge on the line $\theta_1^{-1}(0)\cap\cdots\cap\theta_{n-1}^{-1}(0)$, completing the proof.
  \end{proof}

  Each connected component of the complement of the Stiefel diagram is called a \emph{Weyl alcove}. Since $G$ is simple and of rank $n$, the closure of each Weyl alcove is an $n$-simplex. Let $\alpha_1,\ldots,\alpha_n$ be simple roots and $\widetilde{\alpha}$ be the highest root of $G$, and let $\t$ be the Lie algebra of $T$. We shall consider one of those $n$-simplices
  \[
    \Delta=\{x\in\mathfrak{t}\mid \alpha_1(x)\ge 0,\ldots,\alpha_n(x)\ge 0,\,\widetilde{\alpha}(x)\le 1\}.
  \]
  Let $L$ denote the group generated by coroot shifts. Since $G$ is simply-connected, $L$ is identified with the integer lattice of $\t$. The affine Weyl group of $G$ is defined by
  \[
    W_\mathrm{aff}=W\rtimes L.
  \]
  Then $W_\mathrm{aff}$ acts on $\t$. Since this action fixes the Stiefel diagram, $W_\mathrm{aff}$ permutes Weyl alcoves. By \cite[Theorem 4.5, Part I]{Hu}, we have:

  \begin{lemma}
    \label{alcove permutation}
    The affine Weyl group $W_\mathrm{aff}$ permutes Weyl alcoves of $G$ simply transitively.
  \end{lemma}

  Let $\mathcal{P}$ be the union of all closures of Weyl alcoves around the origin. Then $\mathcal{P}$ is a simplicial convex $n$-polytope.

  \begin{lemma}
    \label{P face}
    If $\sigma$ is a face of $\mathcal{P}$ such that $\sigma+a$ is also a face of $\mathcal{P}$ for some $0\ne a\in L$, then both $\sigma$ and $\sigma+a$ must be faces of the boundary of $\mathcal{P}$.
  \end{lemma}

  \begin{proof}
    Since $W$ permutes Weyl chambers simply transitively, it follows from Lemma \ref{alcove permutation} that $\sigma$ and $\sigma+a$ must not include the vertex $0\in\mathcal{P}$, completing the proof.
  \end{proof}

  By Lemma \ref{P face}, we can define
  \[
    \mathcal{Q}=\mathcal{P}/\sim
  \]
  where $\sigma\sim\sigma+a$ if $\sigma$ and $\sigma+a$ are faces of the boundary of $P$ for some $a\in L$. Clearly, the inclusion $\mathcal{P}\to\t$ induces a homeomorphism
  \begin{equation}
    \label{homeo Q}
    \mathcal{Q}\xrightarrow{\cong}\t/L.
  \end{equation}
  Since $G$ is simply-connected, $L$ is the integer lattice of $\t$, so a torus $\t/L$ coincides with a maximal torus $T$. Then we get:

  \begin{proposition}
    \label{triangulation}
    The homeomorphism \eqref{homeo Q} is a $W$-equivariant triangulation of $T$.
  \end{proposition}

  \begin{proof}
    The homeomorphism \eqref{homeo Q} is obviously $W$-equivariant, so it remains to prove $\mathcal{Q}$ is a simplicial complex. It suffices to show that there is no vertex $v$ of the boundary of $\mathcal{P}$ such that $v$ and $v+a$ are joined by an edge of $\mathcal{P}$ for $0\ne a\in L$. This has been already proved in Lemma \ref{Stiefel diagram}.
  \end{proof}

  \begin{proposition}
    \label{T/W}
    The quotient space $T/W$ is naturally identified with $\Delta$.
  \end{proposition}

  \begin{proof}
    This follows from \cite[Theorem 4.8, Part I]{Hu}.
  \end{proof}

  A maximal torus $T$ will always be equipped with a $W$-equivariant triangulation \eqref{homeo Q}. We consider objects related to the triangulation of $T$. Let $\sigma$ be a face of $T$. Then there is a face $\widetilde{\sigma}$ of $\mathcal{Q}$ which is mapped onto $\sigma$. We can associate to $\widetilde{\sigma}$ roots corresponding to walls including $\widetilde{\sigma}$. Lifts of $\sigma$ are related by translation by $L$, so that the associated roots are equal. Then we can associate roots to $\sigma$. Let $\Phi(\sigma)$ denote the set of roots associated to $\sigma$. Clearly, we have:

  \begin{lemma}
    \label{Phi}
    If faces $\sigma,\tau$ of $T$ satisfy $\sigma<\tau$, then
    \[
      \Phi(\sigma)\supset\Phi(\tau).
    \]
  \end{lemma}

  Let $\sigma$ be a face of $T$. We define two groups associated to $\sigma$. Let $W(\sigma)$ be a subgroup of $W$ generated by reflections corresponding to roots in $\Phi(\sigma)$, and let
  \[
    Z(\sigma)=\{x\in G\mid xy=yx\text{ for each }y\in\sigma\}.
  \]
  We have $W(\sigma)=1$ and $Z(\sigma)=T$ for $\dim\sigma=n$. Notice that since we may assume $\Delta$ is a face of $T$, we can consider the groups $W(\sigma)$ and $Z(\sigma)$ for a face $\sigma$ of $\Delta$.

  \begin{lemma}
    \label{W Z}
    If faces $\sigma,\tau$ of $T$ satisfy $\sigma<\tau$, then
    \[
      W(\sigma)>W(\tau)\quad\text{and}\quad Z(\sigma)>Z(\tau).
    \]
  \end{lemma}

  \begin{proof}
    The first statement follows from Lemma \ref{Phi}. Since $Z(\sigma)$ is a union of all maximal tori including $\sigma$, the second statement is true.
  \end{proof}


  \section{Homotopy decomposition}\label{Homotopy decomposition}

  This section proves a new homotopy decomposition of $\Hom(\Z^m,G)_1$ (Theorem \ref{homotopy decomposition}). Let $\pi$ be a discrete group, and let $\Rep(\pi,G)$ be the quotient of the conjugation action of $G$ on $\Hom(\pi,G)$ as in Section \ref{Introduction}. Then we have
  \[
    \Rep(\Z,G)=\Hom(\Z,G)/G=T/W
  \]
  which is identified with $\Delta$ by Proposition \ref{T/W}. We will consider the composite
  \[
    \pi\colon\Hom(\Z^m,G)_1\to\Hom(\Z,G)_1=\Hom(\Z,G)\to\Rep(\Z,G)=\Delta,
  \]
  where the first map is the $m$-th projection and the second map is the quotient map. We aim to identify the fibers of the map $\pi$. We consider a map
  \[
    \phi\colon G/T\times T^m\to\Hom(\Z^m,G)_1,\quad(gT,t_1,\ldots,t_m)\mapsto(gt_1g^{-1},\ldots,gt_mg^{-1}).
  \]
  It is proved in \cite{Bo} that the map $\phi$ is surjective. Let the Weyl group $W$ act on $G/T\times T^m$ by
  \[
    (gT,t_1,\ldots,t_m)\cdot w=(gwT,w^{-1}t_1w,\ldots,w^{-1}t_mw)
  \]
  for $(gT,t_1,\ldots,t_m)\in G/T\times T^m$ and $w\in W$. Then the map $\phi$ is invariant under the action of $W$, and so it induces a surjective map
  \begin{equation}
    \label{phi}
    G/T\times_WT\to\Hom(\Z^m,G)_1.
  \end{equation}

  \begin{lemma}
    If $x,y\in\Delta$ belong to the interior of a common face, then
    \[
      \pi^{-1}(x)\cong\pi^{-1}(y).
    \]
  \end{lemma}

  \begin{proof}
    Suppose that $x,y\in\Delta$ belong to the interior of a common face. Consider the adjoint action of $G$ on $T^m\times\t$. Then for each $(t_1,\ldots,t_m)\in T^m$, the isotropy subgroups of $(t_1,\ldots,t_m,x)$ and $(t_1,\ldots,t_m,y)$ are equal, so that
    \[
      (\phi\circ\pi)^{-1}(x)=G/T\times T^{m-1}\times W\cdot x\quad\text{and}\quad(\phi\circ\pi)^{-1}(y)=G/T\times T^{m-1}\times W\cdot y.
    \]
    Thus by the definition of the map $\phi$, we obtain $\pi^{-1}(x)\cong\pi^{-1}(y)$, as stated.
  \end{proof}

  Let $P(\Delta)$ denote the face poset of $\Delta$, and let $\sigma_0$ denote the barycenter of a face $\sigma\in P(\Delta)$. For $\sigma\in P(\Delta)$, let $\phi_\sigma\colon G/T\times T^{m-1}\to\pi^{-1}(\sigma_0)$ denote the restriction of the quotient map \eqref{phi}. Observe that for $\sigma<\tau\in P(\Delta)$, there is a natural map $q_{\sigma,\tau}\colon\pi^{-1}(\tau_0)\to\pi^{-1}(\sigma_0)$ satisfying a commutative diagram
  \[
    \xymatrix{
    G/T\times T^{m-1}\ar[r]^(.6){\phi_\tau}\ar[d]_{\phi_\sigma}&\pi^{-1}(\tau_0)\ar[d]^{q_{\sigma,\tau}}\\
    \pi^{-1}(\sigma_0)\ar@{=}[r]&\pi^{-1}(\sigma_0).
    }
  \]
  Clearly, we have
  \[
    q_{\sigma,\tau}\circ q_{\tau,\nu}=q_{\sigma,\nu}
  \]
  for $\sigma<\tau<\nu\in P(\Delta)$. Let $\iota_{\tau,\sigma}\colon\sigma\to\tau$ denote the inclusion for $\sigma<\tau\in P(\Delta)$. Then the above observation implies that
  \begin{equation}
    \label{Hom coprod}
    \Hom(\Z^m,G)_1=\left(\coprod_{\sigma\in P(\Delta)}\pi^{-1}(\sigma_0)\times\sigma\right)/\sim
  \end{equation}
  where $(x,\iota_{\tau,\sigma}(y))\sim(q_{\sigma,\tau}(x),y)$ for $(x,y)\in\pi^{-1}(\tau_0)\times\sigma\subset\pi^{-1}(\tau_0)\times\tau$. Define a functor
  \[
    F_m\colon P(\Delta)\to\mathbf{Top},\quad\sigma\mapsto\pi^{-1}(\sigma_0)
  \]
  such that $F_m(\tau>\sigma)=q_{\sigma,\tau}$, where we understand a poset $P$ as a category by assuming an inequality $x>y\in P$ as a unique morphism $x\to y$. Thus by \eqref{Hom coprod}, we obtain:

  \begin{theorem}
    \label{hocolim}
    There is a homeomorphism
    \[
      \Hom(\Z^m,G)_1\cong\hocolim F_m.
    \]
  \end{theorem}

  We further look into $F_m(\sigma)$ for $\sigma\in P(\Delta)$. Let $P(X)$ denote the face poset of a regular CW complex $X$, and recall that $T$ is equipped with the triangulation \eqref{homeo Q}. Let $\sigma$ be a face of $\Delta$. For $\tau=\tau_1\times\cdots\times\tau_{m-1}\in P(T^{m-1})$, let $Z(\tau)=Z(\tau_1)\cap\cdots\cap Z(\tau_{m-1})$. For $\tau<\mu\in P(T^{m-1})$, let $\iota_{\mu,\tau}\colon\tau\to\mu$ denote the inclusion, and let
  \[
    q_{\tau,\mu}^\sigma\colon G/Z(\mu)\cap Z(\sigma)\to G/Z(\tau)\cap Z(\sigma)
  \]
  be the natural projection. Then we have
  \[
    \pi^{-1}(\sigma_0)=\left(\left(\coprod_{\tau\in P(T^{m-1})}G/Z(\tau)\cap Z(\sigma)\times\tau\right)/\sim\right)/W(\sigma)
  \]
  where $(x,\iota_{\mu,\tau}(y))\sim(q_{\tau,\mu}^\sigma(x),y)$ for $(x,y)\in G/Z(\mu)\cap Z(\sigma)\times\tau\subset G/Z(\mu)\cap Z(\sigma)\times\mu$ and the quotient by $W(\sigma)$ is taken by the action of $W$ on $G/T\times T^{m-1}$. Now we define a functor
  \[
    F_m^\sigma\colon P(T^{m-1})\to\mathbf{Top},\quad\tau\mapsto G/Z(\tau)\cap Z(\sigma)
  \]
  where $F_m^\sigma(\mu>\tau)$ is the projection $q^\sigma_{\tau,\mu}$. Then we obtain:

  \begin{proposition}
    \label{fiber hocolim}
    For $\sigma\in P(\Delta)$, there is a natural homeomorphism
    \[
      F_m(\sigma)\cong(\hocolim F_m^\sigma)/W(\sigma).
    \]
  \end{proposition}

  Hereafter, we let
  \[
    q_m=d+n(m-2).
  \]
  Since the maximal dimension of $F_m^\sigma(\tau)$ is $q_m$, we get:

  \begin{corollary}
    \label{dim fiber}
    For each $\sigma\in P(\Delta)$, $\dim F_m(\sigma)=q_m$.
  \end{corollary}

  \begin{example}
    \label{example hocolim}
    We examine Theorem \ref{hocolim} in the $G=SU(2)$ case. Since $\rank SU(2)=1$, $\Delta$ is a 1-simplex. Let $v_0,v_1$ be vertices of $\Delta$, and let $e$ be an edge of $\Delta$. Since $G=SU(2)$, $\{v_0,v_1\}$ corresponds to the center, so we have
    \[
      F_m(v_i)=\Hom(\Z^{m-1},SU(2))_1
    \]
    for $i=0,1$. By Proposition \ref{fiber hocolim}, we also have
    \[
      F_m(e)=SU(2)/T\times T^{m-1}=S^2\times(S^1)^{m-1}.
    \]
    Then Theorem \ref{hocolim} for $G=SU(2)$ is equivalent to that there is a homotopy pushout
    \[
      \xymatrix{
        S^2\times(S^1)^{m-1}\ar[r]^(.42){g_m}\ar[d]_{g_m}&\Hom(\Z^{m-1},SU(2))_1\ar[d]\\
        \Hom(\Z^{m-1},SU(2))_1\ar[r]&\Hom(\Z^m,SU(2))_1.
      }
    \]
    For $m=2$, the map $g_2\colon S^2\times S^1\to S^3$ is of degree 2. On the other hand, the map $\Hom(\Z^{m-1},SU(2))_1\to\Hom(\Z^m,SU(2))_1$ has a retraction. Then the homotopy pushout above for $m=2$ splits after a suspension, and so we get a stable homotopy equivalence
    \begin{equation}
      \label{stable splitting}
      \Hom(\Z^2,SU(2))_1\underset{\mathrm{s}}{\simeq}S^2\vee S^3\vee(S^3\cup_2 e^4)
    \end{equation}
    which was previously proved by Baird, Jeffrey and Selick \cite{BJS} and Crabb \cite{C} in different ways.
  \end{example}


  \section{The functor $\widehat{F}_m$}\label{the functor}

  This section defines a functor $\widehat{F}_m\colon P(\Delta)\to\mathbf{Top}$ which extracts the top line of (a variant of) the Bousfield-Kan spectral sequence for $\hocolim F_m$. First, we compute some homology of $F_m(\sigma)$. To this end, we recall the work of Baird \cite{B} on the cohomology of $\Hom(\Z^m,G)_1$ over a field whose characteristic does not divide the order of $W$. Let $K$ be a topological group acting on a space $X$, and let $f\colon X\to Y$ be a $K$-equivariant map, where $K$ acts trivially on $Y$. Let $\bar{f}\colon X/K\to Y$ be the induced map from $f$. Baird \cite{B} defined that a map $f\colon X\to Y$ is a \emph{$\mathbb{F}$-cohomological principal $K$-bundle} if $f$ is a closed surjection and
  \[
    \widetilde{H}^*(\bar{f}^{-1}(y);\mathbb{F})=0
  \]
  for each $y\in Y$, where $\mathbb{F}$ is a field. The main result of Baird's work \cite{B} is the following.

  \begin{theorem}
    \label{Baird}
    Let $K$ be a finite group, and let $\mathbb{F}$ be a field of characteristic prime to $|K|$. If a map $f\colon X\to Y$ is a cohomological principal $K$-bundle over $\mathbb{F}$ where $X$ is paracompact and Hausdorff, then there is an isomorphism
    \[
      H^*(X;\mathbb{F})^K\cong H^*(Y;\mathbb{F}).
    \]
  \end{theorem}

  This theorem is applicable to $\Hom(\Z^m,G)_1$ as in \cite{B}.

  \begin{theorem}
    \label{cohomology Hom}
    Let $\mathbb{F}$ be a field of characteristic prime to $|W|$. Then the map \eqref{phi} is a $\mathbb{F}$-cohomological principal $W$-bundle, so that there is an isomorphism
    \[
      H^*(\Hom(\Z^m,G)_1;\mathbb{F})\cong H^*(G/T\times T^m;\mathbb{F})^W.
    \]
  \end{theorem}

  We also apply Theorem \ref{Baird} to $F_m(\sigma)$. The following lemma is immediate from the definition of a cohomological principal bundle.

  \begin{lemma}
    \label{bundle pullback}
    If $f\colon X\to Y$ is a $\mathbb{F}$-cohomological principal $K$-bundle, then for any closed subset $Z\subset Y$, the natural map
    \[
      f^{-1}(Z)\to Z
    \]
    is a $\mathbb{F}$-cohomological principal $K$-bundle.
  \end{lemma}

  We consider special representations of $W$.

  \begin{lemma}
    \label{W-rep G/T}
    \begin{enumerate}
      \item The $W$-representation $H^n(T;\Q)$ is the sign representation.

      \item For $n\ge 2$, the $W$-representation $H^{n-1}(T;\Q)$ does not include the trivial representation.

      \item The $W$-representation $H^{\dim G-n}(G/T;\Q)$ is the sign representation.
    \end{enumerate}
  \end{lemma}

  \begin{proof}
    (1) Since each reflection of $W$ changes the orientation of $\t$ and $H^1(T;\R)\cong\t$ as a $W$-module, $H^n(T;\Q)\cong\Lambda^nH^1(T;\Q)$ is the sign representation of $W$.

    \noindent(2) By \cite[Theorem III.2.4]{Br}, there is an isomorphism
    \[
      H^*(T/W;\Q)\cong H^*(T;\Q)^W.
    \]
    Then by Proposition \ref{T/W}, $H^{n-1}(T;\Q)^W=0$ for $n\ge 2$, completing the proof.

    \noindent(3) By Theorem \ref{Baird}, there is an isomorphism
    \[
      H^*(G/T\times T;\Q)^W\cong H^*(G;\Q)
    \]
    because $\Hom(\Z,G)_1=G$. Then we get
    \[
      (H^{d-n}(G/T;\Q)\otimes H^n(T;\Q))^W\cong H^d(G/T\times T;\Q)^W\cong H^d(G;\Q)\cong\Q.
    \]
    So since $H^{d-n}(G/T;\Q)\cong H^n(T;\Q)\cong\Q$, $H^{d-n}(G/T;\Q)\otimes H^n(T;\Q)$ is the trivial $W$-representation. Thus the statement follows from (1).
  \end{proof}

  Now we compute the homology of $F_m(\sigma)$.

  \begin{lemma}
    \label{top homology fiber}
    For $\sigma\in\Delta$, we have
    \[
      H_{q_m}(F_m(\sigma))\cong
      \begin{cases}
        \Z&m\text{ is even or }\dim\sigma=n\\
        0&m\text{ is odd and }\dim\sigma<n.
      \end{cases}
    \]
  \end{lemma}

  \begin{proof}
    By Corollary \ref{dim fiber}, $H_{q_m}(F_m(\sigma))$ is a free abelian group. Then we compute $\dim H^{q_m}(F_m(\sigma);\Q)$ because $\rank H_{q_m}(F_m(\sigma))=\dim H^{q_m}(F_m(\sigma);\Q)$. By Theorem \ref{cohomology Hom} and Lemma \ref{bundle pullback}, the map
    \[
      \phi^{-1}(\pi^{-1}(\sigma_0))\to\pi^{-1}(\sigma_0)
    \]
    is a $\Q$-cohomological principal $W$-bundle. The space $\phi^{-1}(\pi^{-1}(\sigma_0))$ has $|W|/|W(\sigma)|$ connected components and each component is homeomorphic with $G/T\times T^{m-1}$. Moreover, $W$ permutes these components transitively and each component is fixed by the action of $W(\sigma)$. Then the map
    \[
      G/T\times T^{m-1}\to\pi^{-1}(\sigma_0)=F_m(\sigma)
    \]
    is a $\Q$-cohomological principal $W(\sigma)$-bundle. Thus by Theorem \ref{Baird}, we obtain an isomorphism
    \[
      H^*(F_m(\sigma);\Q)\cong H^*(G/T\times T^{m-1};\Q)^{W(\sigma)}.
    \]
    If $\dim\sigma=n$, then $W(\sigma)=1$, implying $\dim H^{q_m}(F_m(\sigma);\Q)=1$. Now we assume $\dim\sigma<n$, or equivalently, $W(\sigma)\ne 1$. By Lemma \ref{W-rep G/T}, $H^{d-n}(G/T;\Q)$ and $H^n(T;\Q)$ are the sign representation of $W$. Then it follows from the K\"unneth theorem that $H^{q_m}(G/T\times T^{m-1};\Q)$ is the tensor product of $m$ copies of the sign representation of $W$. Thus since $W(\sigma)\ne 1$, we obtain
    \[
      H^{q_m}(G/T\times T^{m-1};\Q)^{W(\sigma)}\cong
      \begin{cases}
        \Q&m\text{ is even}\\
        0&m\text{ is odd.}
      \end{cases}
    \]
    Therefore the proof is complete.
  \end{proof}

  Now we define a functor $\widehat{F}_m\colon P(\Delta)\to\mathbf{Top}$ by
  \[
    \widehat{F}_m(\sigma)=
    \begin{cases}
      S^{q_m}&m\text{ is even or }\dim\sigma=n\\
      *&m\text{ is odd and }\dim\sigma<n
    \end{cases}
  \]
  such that the map $F_m(\sigma>\tau)$ is the constant map for $m$ odd and a map of degree $|W(\tau)|/|W(\sigma)|$ for $m$ even. Since
  \[
    (|W(\mu)|/|W(\tau)|)\cdot(|W(\tau)|/|W(\sigma)|)=|W(\mu)|/|W(\sigma)|
  \]
  for $\sigma>\tau>\mu\in P(\Delta)$, $\widehat{F}_m$ is well-defined.

  Next we define a natural transformation $\rho\colon F_m\to\widehat{F}_m$. For $m$ odd, $\rho$ is defined by the pinch map onto the top cell $G/T\times T^{m-1}\to S^{q_m}$ and the constant map. Suppose $m$ is even. For $\sigma\in P(\Delta)$, let $\mathcal{Q}(\sigma)$ be the union of the boundary of $\mathcal{P}$ and the image of all walls of the Stiefel diagram including $\sigma$ under the projection $\t\to\mathcal{Q}$, where $\mathcal{Q}$ is the triangulation of $T$ in Section \ref{Triangulation of a maximal torus} and $\t$ is the Lie algebra of $T$. Then by \eqref{homeo Q},
  \[
    \mathcal{Q}/\mathcal{Q}(\sigma)=\bigvee_{|W(\sigma)|}S^n.
  \]
  Moreover, for $\sigma>\tau\in P(\Delta)$, $\mathcal{Q}(\sigma)\subset\mathcal{Q}(\tau)$, implying there is a commutative diagram
  \begin{equation}
    \label{pinch map}
    \xymatrix{
      \mathcal{Q}/\mathcal{Q}(\sigma)\ar@{=}[r]\ar[d]&\bigvee_{|W(\sigma)|}S^n\ar[d]^{\bigvee_{|W(\sigma)|}\nabla}\\
      \mathcal{Q}/\mathcal{Q}(\tau)\ar@{=}[r]&\bigvee_{|W(\sigma)|}\bigvee_{|W(\tau)|/|W(\sigma)|}S^n\ar@{=}[r]&\bigvee_{|W(\tau)|}S^n,
    }
  \end{equation}
  where $\nabla\colon S^n\to\bigvee_{|W(\tau)|/|W(\sigma)|}S^n$ is the pinch map. On the other hand, a face $\tau$ of $\mathcal{Q}$ satisfies
  \[
    Z(\tau)\cap Z(\sigma)=T
  \]
  whenever $\Int(\tau)$ is in $\mathcal{Q}-\mathcal{Q}(\sigma)$, where $Z(\tau)\cap Z(\sigma)$ always includes $T$. Then by Proposition \ref{fiber hocolim}, there is a projection
  \begin{align*}
    F_m(\sigma)\to&((G/T\times T^{m-2})\wedge(\mathcal{Q}/\mathcal{Q}(\sigma)))/W(\sigma)\\
    &=\left((G/T\times T^{m-2})\wedge\bigvee_{|W(\sigma)|}S^n\right)/W(\sigma).
  \end{align*}
  Since $W(\sigma)$ permutes spheres in $\bigvee_{|W(\sigma)|}S^n$,
  \[
    \left((G/T\times T^{m-2})\wedge\bigvee_{|W(\sigma)|}S^n\right)/W(\sigma)=(G/T\times T^{m-2})\wedge S^n
  \]
  By \eqref{pinch map}, this map satisfies the commutative diagram
  \[
    \xymatrix{
      F_m(\sigma)\ar[r]\ar[d]_{F_m(\sigma>\tau)}&(G/T\times T^{m-2})\wedge S^n\ar[d]^{|W(\tau)|/|W(\sigma)|}\\
      F_m(\tau)\ar[r]&(G/T\times T^{m-2})\wedge S^n.
    }
  \]
  Thus composing with the pinch map onto the top cell $(G/T\times T^{m-2})\wedge S^n\to S^{q_m}$, we obtain a natural transformation $\rho\colon F_m\to\widehat{F}_m$.

  We show properties of the natural transformation $\rho\colon F_m\to\widehat{F}_m$ in homology. By the construction and Lemma \ref{top homology fiber}, we have:

  \begin{proposition}
    \label{top homology projection}
    Let $\sigma \in P(\Delta)$. If $m$ is even or $\dim \sigma=n$, then the map $\rho_\sigma\colon F_m(\sigma)\to\widehat{F}_m(\sigma)$ is an isomorphism in $H_{q_m}$ .
  \end{proposition}

  The following variant of the Bousfield-Kan spectral sequence for a homotopy colimit is constructed in \cite{HKTT}. See \cite[XII 4.5]{BK} for the original Bousfield-Kan spectral sequence.

  \begin{proposition}
    \label{spectral sequence}
    Let $F\colon P(K)\to\mathbf{Top}$ be a functor, where $P(K)$ denotes the face poset of a simplicial complex $K$. Then there is a spectral sequence
    \[
      E^1_{p,q}=\bigoplus_{\sigma\in P_p(K)}H_q(F(\sigma))\quad\Longrightarrow\quad H_{p+q}(\hocolim F),
    \]
    where $P_p(K)$ denotes the set of $p$-simplices of $K$.
  \end{proposition}

  By Proposition \ref{top homology projection}, we get:

  \begin{lemma}
    \label{Bousfield-Kan}
    Let $E^r$ and $\widehat{E}^r$ be the spectral sequences of Proposition \ref{spectral sequence} for $\hocolim F_m$ and $\hocolim\widehat{F}_m$, respectively. Then the natural transformation $\rho\colon F_m\to\widehat{F}_m$ induces an isomorphism of the top lines
    \[
      \rho_*\colon E_{*,q_m}^1\xrightarrow{\cong}\widehat{E}_{*,q_m}^1.
    \]
  \end{lemma}

  \begin{proposition}
    \label{direct summand}
    $H_*(\hocolim\widehat{F}_m)$ is a direct summand of $H_*(\Hom(\Z^m,G)_1)$ for $*\ge q_m$.
  \end{proposition}

  \begin{proof}
    Let $(E^r,d^r)$ and $(\widehat{E}^r,\hat{d}^r)$ denote the spectral sequences of Proposition \ref{spectral sequence} for $\hocolim F_m$ and $\hocolim\widehat{F}_m$, respectively. Let $r$ be the smallest integer $\ge 2$ such that there is a non-trivial differential $d^r_{p,q_m-r+1}\colon E^r_{p,q_m-r+1}\to E^r_{p-r,q_m}$ for some $p\ge 0$. Suppose that $d^r_{p,q_m-r+1}(x)\ne 0$ for $x\in E^r_{p,q_m-r+1}$. Then $E^r$ and $\widehat{E}^r$ are illustrated below, where possibly non-trivial parts are shaded.

    
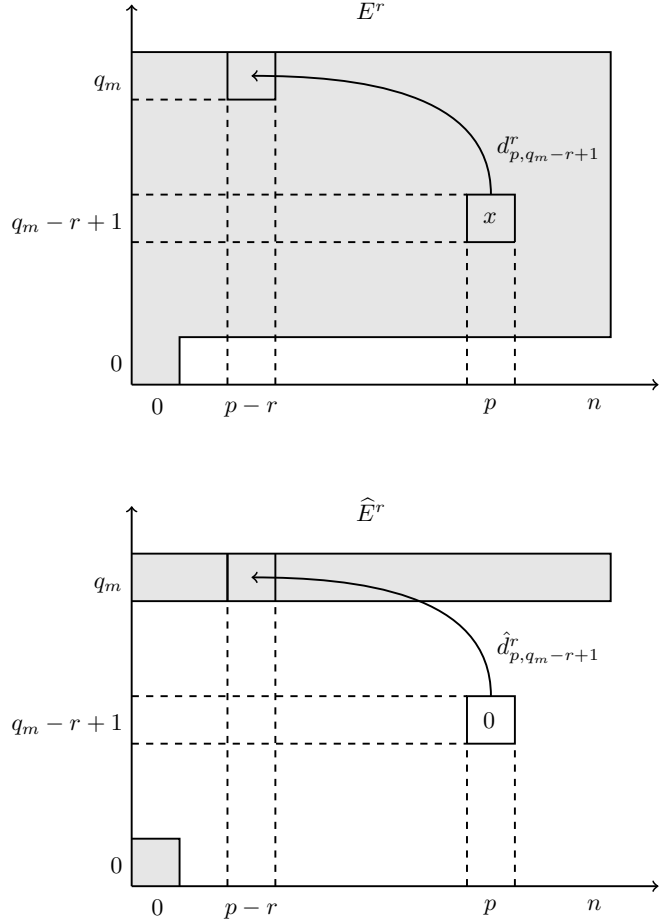
\begin{figure}[H]
      \centering
      \scalebox{0.9}{
      \begin{tikzpicture}[x=0.7cm, y=0.7cm, thick]
        \fill[black!10!white](0,0)--(1,0)--(1,1)--(10,1)--(10,7)--(0,7)--(0,0);
        \draw[->](0,0)--(11,0);
        \draw[->](0,0)--(0,8);
        \draw(1,0)--(1,1)--(10,1)--(10,7)--(0,7);
        \draw(2,7)--(2,6)--(3,6)--(3,7);
        \draw(8,3)--(7,3)--(7,4)--(8,4)--(8,3);
        \draw(7.5,4)[->, out=90, in=0] to(2.5,6.5);
        \draw[dashed](0,6)--(2,6);
        \draw[dashed](7,0)--(7,3);
        \draw[dashed](8,0)--(8,3);
        \draw[dashed](2,0)--(2,6);
        \draw[dashed](3,0)--(3,6);
        \draw[dashed](0,3)--(7,3);
        \draw[dashed](0,4)--(7,4);
        \draw(0,6) node[above left]{$q_m$};
        \draw(0,3) node[above left]{$q_m-r+1$};
        \draw(0,0.1) node[above left]{$0$};
        \draw(0.85,-0.1) node[below left]{$0$};
        \draw(10,-0.1) node[below left]{$n$};
        \draw(7.8,-0.1) node[below left]{$p$};
        \draw(2.5,-0.1) node[below]{$p-r$};
        \draw(7.15,3.8) node[below right]{$x$};
        \draw(8.7,4.5) node[above]{$d^r_{p,q_m-r+1}$};
        \draw(5,7.5) node[above]{$E^r$};
      \end{tikzpicture}
      }
    \end{figure}

    \begin{figure}[H]
      \centering
      \scalebox{0.9}{
      \begin{tikzpicture}[x=0.7cm, y=0.7cm, thick]
        \fill[black!10!white](0,0)--(1,0)--(1,1)--(0,1)--(0,0);
        \fill[black!10!white](0,7)--(10,7)--(10,6)--(0,6)--(0,7);
        \draw[->](0,0)--(11,0);
        \draw[->](0,0)--(0,8);
        \draw(1,0)--(1,1)--(0,1);
        \draw(0,7)--(10,7)--(10,6)--(0,6);
        \draw(2,7)--(2,6);
        \draw(3,7)--(3,6);
        \draw(2,7)--(2,6);
        \draw(7,3)--(8,3)--(8,4)--(7,4)--(7,3);
        \draw(7.5,4)[->, out=90, in=0] to(2.5,6.5);
        \draw[dashed](7,0)--(7,3);
        \draw[dashed](8,0)--(8,3);
        \draw[dashed](2,0)--(2,6);
        \draw[dashed](3,0)--(3,6);
        \draw[dashed](0,3)--(7,3);
        \draw[dashed](0,4)--(7,4);
        \draw(0,6) node[above left]{$q_m$};
        \draw(0,3) node[above left]{$q_m-r+1$};
        \draw(0,0.1) node[above left]{$0$};
        \draw(0.85,-0.1) node[below left]{$0$};
        \draw(10,-0.1) node[below left]{$n$};
        \draw(7.8,-0.1) node[below left]{$p$};
        \draw(2.5,-0.1) node[below]{$p-r$};
        \draw(7.15,3.85) node[below right]{$0$};
        \draw(8.7,4.5) node[above]{$\hat{d}^r_{p,q_m-r+1}$};
        \draw(5,7.5) node[above]{$\widehat{E}^r$};
      \end{tikzpicture}
      }
    \end{figure}

    By Lemma \ref{Bousfield-Kan}, the natural map $\rho_*\colon E_{p,q_m}^r\to\widehat{E}_{p,q_m}^r$ is an isomorphism, implying
    \[
      0\ne\rho_*(d^r_{p,q_m-r+1}(x))=\widehat{d}^r_{p,q_m-r+1}(\rho_*(x))=0.
    \]
    This is a contradiction. Thus we obtain $E^2_{p,q_m}\cong E^\infty_{p,q_m}$. On the other hand, we have $\widehat{E}^2_{p,q_m}\cong\widehat{E}^\infty_{p,q_m}\cong H_{*+q_m}(\hocolim\widehat{F}_m)$. Then the composite
    \[
      E^\infty_{p,q_m}\to H_{p+q_m}(\hocolim F_m)\xrightarrow{\rho_*}H_{p+q_m}(\hocolim\widehat{F}_m)\cong\widehat{E}^\infty_{p,q_m}
    \]
    is identified with $\rho_*\colon E^2_{p,q_m}\to\widehat{E}^2_{p,q_m}$, implying it is an isomorphism. Therefore by Theorem \ref{hocolim}, the proof is complete.
  \end{proof}

  \begin{example}
    We examine $\hocolim\widehat{F}_m$ for $G=SU(2)$. In this case, $\Delta$ is a 1-simplex, and so as in Example \ref{example hocolim}, there is a homotopy pushout involving $\hocolim\widehat{F}_m$ which yields a homotopy equivalence
    \[
      \hocolim\widehat{F}_m\simeq
      \begin{cases}
        S^{m+2}&m\text{ is odd}\\
        S^{m+1}\vee(S^{m+1}\cup_2e^{m+2})&m\text{ is even.}
      \end{cases}
    \]
    In particular, we can see from \eqref{stable splitting} that $\hocolim\widehat{F}_m$ computes the top homology of $\Hom(\Z^m,SU(2))_1$. This will be generalized in the next section to an arbitrary simply-connected simple $G$.
  \end{example}


  \section{Top homology}\label{Top homology}

  This section computes the top homology of $\Hom(\Z^m,G)_1$ and proves Theorem \ref{main top}. The result depends on the parity of $m$. We start with the case $m$ is odd.

  \begin{theorem}
    \label{top homology odd}
    If $m$ is odd, then the top homology of $\Hom(\Z^m,G)_1$ is
    \[
      H_{q_m+n}(\Hom(\Z^m,G)_1)\cong\Z.
    \]
  \end{theorem}

  \begin{proof}
    We present two proofs.

    \noindent\textbf{First proof.} By Corollary \ref{dim fiber}, the $E^1$-term of the spectral sequence of Proposition \ref{spectral sequence} for $\hocolim F_m$ is given below, where a possibly non-trivial part is shaded. Then by degree reasons, the statement is proved.
    
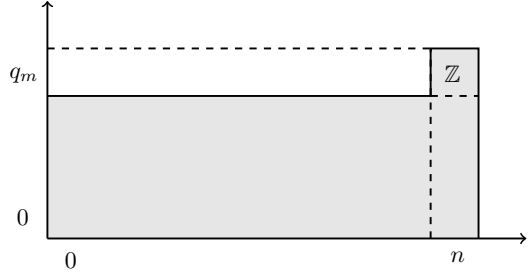
\begin{figure}[htbp]
      \scalebox{0.9}{
      \begin{tikzpicture}[x=0.7cm, y=0.7cm, thick]
        \fill[black!10!white](0,0)--(0,3)--(8,3)--(8,4)--(9,4)--(9,0)--(0,0);
        \draw[->](0,0)--(10,0);
        \draw[->](0,0)--(0,5);
        \draw(0,3)--(8,3)--(8,4)--(9,4)--(9,0);
        \draw[dashed](0,4)--(8,4);
        \draw[dashed](8,3)--(9,3);
        \draw[dashed](8,0)--(8,4);
        \draw(8.1,3.1) node[above right]{$\Z$};
        \draw(0,3.1) node[above left]{$q_m$};
        \draw(-0.2,0.1) node[above left]{$0$};
        \draw(0.8,-0.1) node[below left]{$0$};
        \draw(8.9,-0.1) node[below left]{$n$};
      \end{tikzpicture}
      }
    \end{figure}

    \noindent\textbf{Second proof.} $\Hom(\Z^m,G)_1$ is of dimension $q_m+n$ as mentioned above, implying $H_{q_m+n}(\Hom(\Z^m,G)_1)$ is a free abelian group. Then it suffices to compute the dimension of the rational cohomology $H^{q_m+n}(\Hom(\Z^m,G)_1;\Q)$. By Theorem \ref{cohomology Hom} and the K\"unneth theorem,
    \[
      H^{q_m+n}(\Hom(\Z^m,G)_1;\Q)\cong (H^{\dim G-n}(G/T;\Q)\otimes\underbrace{H^n(T;\Q)\otimes\cdots\otimes H^n(T;\Q)}_m)^W.
    \]
    By Lemma \ref{W-rep G/T}, $H^{\dim G-n}(G/T;\Q)$ and $H^n(T;\Q)$ are the sign representation of $W$, and so we get $\dim H^{q_m+n}(\Hom(\Z^m,G)_1;\Q)=1$, completing the proof.
  \end{proof}

  Next, we consider the case $m$ is even.

  \begin{lemma}
    \label{even top}
    If $m$ is even and $n\ge 2$, then $H_{q_m+n-1}(\Hom(\Z^m,G)_1;\Q)=0$.
  \end{lemma}

  \begin{proof}
    By Lemma \ref{W-rep G/T}, $H^{n-1}(T;\Q)$ does not include the trivial representation of $W$. Then by arguing as in the second proof of Theorem \ref{top homology odd}, the statement is proved.
  \end{proof}

  \begin{remark}
    \label{tophomology SU(2)}
    For $n=1$, $H^{n-1}(T;\Q)$ is the trivial representation of $W$, so that the $n=1$ case is special in Lemma \ref{even top}. The top homology of $\Hom(\Z^m,SU(2))_1$ for $m$ even can be deduced from the results of Baird, Jeffrey and Selick \cite{BJS} and Crabb \cite{C}; it is given by $H_{m+1}(\Hom(\Z^m,G)_1)\cong\Z/2$.
  \end{remark}

  \begin{theorem}
    \label{top homology even}
    If $m$ is even and $n\ge 2$, then the top homology of $\Hom(\Z^m,G)_1$ is
    \[
      H_{q_m+n-1}(\Hom(\Z^m,G)_1)\cong\Z/2.
    \]
  \end{theorem}

  \begin{proof}
    Let $E^r$ and $\widehat{E}^r$ denote the spectral sequences of Proposition \ref{spectral sequence} for $\hocolim F_m$ and $\hocolim\widehat{F}_m$, respectively. By Proposition \ref{fiber hocolim}, $E^1_{n,*}\cong H_*(G/T\times T^{m-1})$. Then by Corollary \ref{dim fiber}, $E^1$ is given below, where a possibly non-trivial part is shaded.

    
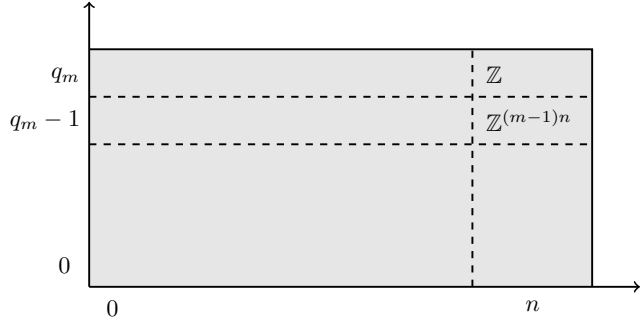
\begin{figure}[H]
      \scalebox{0.9}{
      \begin{tikzpicture}[x=0.7cm, y=0.7cm, thick]
        \fill[black!10!white](0,0)--(0,5)--(10.5,5)--(10.5,0)--(0,0);
        \draw[->](0,0)--(11.5,0);
        \draw[->](0,0)--(0,6);
        \draw(0,5)--(10.5,5)--(10.5,0);
        \draw[dashed](0,4)--(10.5,4);
        \draw[dashed](0,3)--(10.5,3);
        \draw[dashed](8,0)--(8,5);
        \draw(8.1,4.1) node[above right]{$\Z$};
        \draw(8.1,3.1) node[above right]{$\Z^{(m-1)n}$};
        \draw(0,4.1) node[above left]{$q_m$};
        \draw(0,3.1) node[above left]{$q_m-1$};
        \draw(-0.2,0.1) node[above left]{$0$};
        \draw(0.8,-0.1) node[below left]{$0$};
        \draw(9.6,-0.1) node[below left]{$n$};
      \end{tikzpicture}
      }
    \end{figure}

    \noindent Thus $H_*(\Hom(\Z^m,G)_1)=0$ for $*>q_m+n$, and by Corollary \ref{Bousfield-Kan} and Lemma \ref{even top},
    \[
      H_{q_m+n-i}(\Hom(\Z^m,G)_1)\cong E^2_{n-i,q_m}\cong\widehat{E}^2_{n-i,q_m}
    \]
    for $i=0,1$. Let $\sigma$ be the only one $n$-face of $\Delta$, and let $\tau_0,\ldots,\tau_n$ be $(n-1)$-faces of $\Delta$. Then
    \[
      \widehat{E}^1_{n,q_m}=\Z\langle u\otimes\sigma\rangle\quad\text{and}\quad\widehat{E}^1_{n-1,q_m}=\Z\langle u\otimes\tau_i\mid i=0,\ldots,n\rangle,
    \]
    where $u$ is a generator of $H_{q_m}(S^{q_m})$. Since $|W(\sigma)|=1$ and $|W(\tau_i)|=2$ for each $i$, we have
    \[
      d^1\colon \widehat{E}^1_{n,q_m}\to\widehat{E}^1_{n-1,q_m},\quad u\otimes\sigma\mapsto 2\sum_{i=0}^n(-1)^iu\otimes\tau_i,
    \]
    implying $\widehat{E}^2_{n,q_m}=0$, which is proved by the same way as the second proof of Theorem \ref{top homology odd}. Since $\widehat{E}^1_{n-1,q_m}$ is a free abelian group, $\sum_{i=0}^n(-1)^iu\otimes\tau_i\in\widehat{E}^1_{n-1,q_m}$ is a non-trivial cycle. Then by Lemma \ref{even top}, $\widehat{E}^2_{n-1,q_m}\cong\Z/2$, completing the proof.
  \end{proof}

  \begin{proof}
    [Proof of Theorem \ref{main top}]
    Combine Theorems \ref{top homology odd}, \ref{top homology even} and Remark \ref{tophomology SU(2)}.
  \end{proof}


  \section{The complex $\Delta_p(k)$}\label{the complex}

  This section provides a combinatorial way to detect of torsion in the homology of $\Hom(\Z^m,G)_1$. Define a subcomplex of $\Delta$ by
  \[
    \Delta_p(k)=\{\sigma\in\Delta\mid p^{k+1}\text{ does not divide }|W|/|W(\sigma)|\}.
  \]
  Then there is a filtration of subcomplexes
  \[
    \Delta_p(0)\subset\Delta_p(1)\subset\cdots\subset\Delta_p(r)=\Delta,
  \]
  where $r$ is given by $|W|=p^rq$ with $(p,q)=1$.

  \begin{example}
    Let $G=SU(3)$, so that $\Delta$ is a 2-simplex and $|W|=6$. Then possibly non-trivial $\Delta_p(k)$ are $\Delta_2(0)$ and $\Delta_3(0)$. It is easy to see that $\Delta_2(0)$ is the 1-skeleton of $\Delta$ and $\Delta_3(0)$ is the 0-skeleton of $\Delta$.
  \end{example}

  The following lemma will play a fundamental role in connecting the mod $p$ homology of $\Delta_p(k)$ to $p$-torsion in the homology of $\Hom(\Z^2,G)_1$.

  \begin{lemma}
    \label{acyclic}
    $H_*(\hocolim\widehat{F}_2)$ is a finite abelian group for each $*>q_2$ and a finitely generated abelian group of rank 1 for $*=q_2$.
  \end{lemma}

  \begin{proof}
    Since $\hocolim\widehat{F}$ is a CW complex of finite type, the statement is equivalent to that $H_*(\hocolim\widehat{F}_2;\Q)$ is trivial for each $*>q_2$ and isomorphic with $\Q$ for $*=q_2$. Then we compute the rational homology of $\hocolim\widehat{F}_2$. Let $C_*(-;R)$ denote the cellular chain complex over a field $R$. We assume that a sphere $S^{q_2}$ is given a cell decomposition $S^{q_2}=e^0\cup e^{q_2}$. Then for $0\le *\le n$, we can define a map
    \[
      \phi\colon C_*(\Delta;\Q)\to C_{*+q_2}(\hocolim\widehat{F}_2;\Q),\quad\sigma\mapsto \frac{|W(\sigma)|}{|W|}u\times\sigma
    \]
    where $\sigma$ is a face of $\Delta$ and $u$ is a generator of $C_{q_2}(S^{q_2};\Z)$. Then
    \[
      \phi(\partial\sigma)=\sum_{i=0}^{\dim\sigma}(-1)^i\frac{|W(\tau_i)|}{|W|}u\times\tau_i=\frac{|W(\sigma)|}{|W|}\sum_{i=0}^{\dim\sigma}(-1)^i\frac{|W(\tau_i)|}{|W(\sigma)|}u\times\tau_i=\partial\phi(\sigma),
    \]
    where $\partial\sigma=\sum_{i=0}^{\dim\sigma}(-1)^i\tau_i$. Thus $\phi$ is a chain map. Clearly, $\phi$ is bijective, and so $H_*(\Delta;\Q)\cong H_{*+q_2}(\hocolim\widehat{F}_2;\Q)$. Since $\Delta$ is contractible, the proof is done.
  \end{proof}

  Now we state the main theorem of this section.

  \begin{theorem}
    \label{Delta torsion F}
    The mod $p$ homology of $\Delta_p(k)$ is non-trivial for some $k$ if and only if $\hocolim\widehat{F}_2$ has $p$-torsion in homology.
  \end{theorem}

  \begin{proof}
    We assume that a sphere $S^{q_2}$ is given a cell decomposition $S^{q_2}=e^0\cup e^{q_2}$. Consider a map
    \[
      \phi_k\colon C_*(\Delta_p(k);\Z/p)\to C_{*+q_2}(\hocolim\widehat{F}_2;\Z/p),\quad\sigma\mapsto\frac{|W(\sigma)|}{p^{r-k}}u\times\sigma,
    \]
    where $\sigma$ is a face of $\Delta_p(k)$, $u$ is a generator of $C_{q_2}(S^{q_2};\Z/p)\cong\Z/p$ and $r$ is given by $|W|=p^rq$ with $(p,q)=1$. Then
    \[
      \phi_k(\partial\sigma)=\sum_{i=0}^{\dim\sigma}(-1)^i\frac{|W(\tau_i)|}{p^{r-k}}u\times\tau_i=\frac{|W(\sigma)|}{p^{r-k}}\sum_{i=0}^{\dim\sigma}(-1)^i\frac{|W(\tau_i)|}{|W(\sigma)|}u\times\tau_i=\partial\phi_k(\sigma)
    \]
    and so $\phi_k$ is a chain map.

    Let $P(K)$ denote the face poset of a simplicial complex $K$ as in Section \ref{the functor}. Suppose that $\widetilde{H}_*(\Delta_p(k);\Z/p)\ne 0$. Then there is a non-boundary cycle
    \[
      \alpha=\sum_{\sigma\in P(\Delta_p(k))}a_\sigma\sigma
    \]
    in the reduced cellular chain complex $\widetilde{C}_*(\Delta_p(k);\Z/p)$. We may assume that the homology class $[\alpha]$ does not lie in the image of the natural map
    \[
      H_*(\Delta_p(k-1);\Z/p)\to H_*(\Delta_p(k);\Z/p)
    \]
    because we can replace $\Delta_p(k)$ with $\Delta_p(l)$ for some $l<k$ otherwise. Since $\phi_k$ annihilates $C_*(\Delta_p(k-1);\Z/p)$,
    \[
      \phi_k(\alpha)=\sum_{\sigma\in P(\Delta_p(k))-P(\Delta_p(k-1))}\frac{|W(\sigma)|}{p^{r-k}}a_\sigma(u\times\sigma)\in C_{*+q_2}(\hocolim\widehat{F}_2;\Z/p)
    \]
    which is a cycle because $\alpha$ is a cycle and $\phi_k$ is a chain map. Suppose
    \[
      \phi_k(\alpha)=\partial\left(\sum_{\tau\in P(\Delta_p(k))}b_\tau(u\times\tau)\right).
    \]
    Then we have
    \begin{align*}
      \phi_k(\alpha)&=\partial\left(\sum_{\tau\in P(\Delta_p(k))-P(\Delta_p(k-1))}b_\tau(u\times\tau)\right)\\
      &=\partial\phi_k\left(\sum_{\tau\in P(\Delta_p(k))-P(\Delta_p(k-1))}\frac{p^{r-k}}{|W(\tau)|}b_\tau\tau\right).
    \end{align*}
    By definition, $\phi_k$ is injective on the subgroup of $C_*(\Delta_p(k);\Z/p)$ generated by faces in $P(\Delta_p(k))-P(\Delta_p(k-1))$. So we obtain that $\alpha$ is homologous to a cycle in $\Delta_p(k-1)$. Since $\alpha$ is not homologous to a non-boundary cycle in $\Delta_p(k-1)$ by assumption, $\alpha$ is homologous to a boundary in $\Delta_p(k)$. Then $\alpha$ itself is a boundary, which is a contradiction. Then $\phi_k(\alpha)$ is a non-boundary cycle in $C_{*+q_2}(\hocolim\widehat{F}_2;\Z/p)$. Moreover, by definition, $\phi_k(\alpha)$ is not the mod $p$ reduction of a representative of an integral homology class of infinite order. Thus by Lemma \ref{acyclic}, $H_*(\hocolim\widehat{F}_2)$ has $p$-torsion in homology.

    Assume $\widetilde{H}_*(\Delta_p(k);\Z/p)=0$ for each $k$. So we suppose $\hocolim\widehat{F}_2$ has $p$-torsion in homology and derive a contradiction. The $E^2$-term of the spectral sequence of Proposition \ref{spectral sequence} for $\hocolim\widehat{F}_2$ is illustrated as follows, where possibly non-trivial parts are shaded.
    \begin{center}
    \scalebox{0.9}
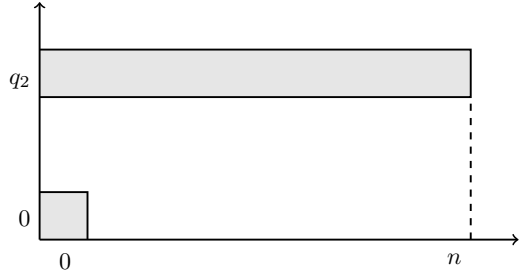
{
    \begin{tikzpicture}[x=0.7cm, y=0.7cm, thick]
      \fill[black!10!white](0,0)--(1,0)--(1,1)--(0,1)--(0,0);
      \fill[black!10!white](0,4)--(9,4)--(9,3)--(0,3)--(0,4);
      \draw[->](0,0)--(10,0);
      \draw[->](0,0)--(0,5);
      \draw(1,0)--(1,1)--(0,1);
      \draw(0,4)--(9,4)--(9,3)--(0,3);
      \draw[dashed](9,0)--(9,3);
      \draw(0,3) node[above left]{$q_2$};
      \draw(0,0.1) node[above left]{$0$};
      \draw(0.85,-0.1) node[below left]{$0$};
      \draw(9,-0.1) node[below left]{$n$};
    \end{tikzpicture}}
    \end{center}
    Then $H_*(\hocolim\widehat{F}_2)$ has $p$-torsion only for $*\ge q_2$. So there is a cycle
    representing a $p$-torsion element in the homology of $\hocolim\widehat{F}_2$. We may assume that its mod $p$ reduction
    \[
      \widehat{\alpha}=\sum_{\sigma\in P(\Delta)}\bar{a}_\sigma(u\times\sigma)
    \]
    represents a non-trivial mod $p$ homology class of $\hocolim\widehat{F}_2$. Let
    \[
      \widehat{\alpha}_k=\sum_{\sigma\in P(\Delta_p(k))-P(\Delta_p(k-1))}\bar{a}_\sigma(u\times\sigma).
    \]
    By definition, $\widehat{\alpha}=\widehat{\alpha}_0+\cdots+\widehat{\alpha}_r$. Clearly, $\partial\widehat{\alpha}_k$ is a linear combination of simplices in $P(\Delta_p(k))$. For $\sigma\in P(\Delta_p(k))-P(\Delta_p(k-1))$ and $\tau\in P(\Delta_p(k-1))$, if $\sigma>\tau$, then $|W(\tau)|/|W(\sigma)|\equiv 0\mod p$, implying that $\partial\widehat{\alpha}_k$ is actually a linear combination of simplices in $P(\Delta_p(k))-P(\Delta_p(k-1))$. Thus since $\partial\widehat{\alpha}=0$, we get
    \[
      \partial\widehat{\alpha}_k=0
    \]
    for each $k$. Let
    \[
      \alpha_k=\sum_{\sigma\in P(\Delta_p(k))-P(\Delta_p(k-1))}\frac{p^{r-k}}{|W(\sigma)|}a_\sigma\sigma\in\widetilde{C}_*(\Delta_p(k);\Z/p).
    \]
    Then $\phi_k(\alpha_k)=\widehat{\alpha}_k$, implying $\partial\phi_k(\alpha_k)=0$. So we get $\partial\alpha_k\in\widetilde{C}_*(\Delta_p(k-1);\Z/p)$. Since $\widetilde{H}_*(\Delta_p(k-1);\Z/p)=0$, there is $\beta_k\in\widetilde{C}_*(\Delta_p(k-1);\Z/p)$ such that $\partial\alpha_k=\partial\beta_k$. Then since the map $\phi_k$ annihilates $C_*(\Delta_p(k-1);\Z/p)$, we get
    \[
      \phi_k(\alpha_k-\beta_k)=\phi_k(\alpha_k)=\widehat{\alpha}_k.
    \]
    So since $\widetilde{H}_*(\Delta_p(k-1);\Z/p)=0$, $\alpha_k-\beta_k$ is a boundary, implying $\widehat{\alpha}_k$ is a boundary. Thus since $k$ is arbitrary, $\widehat{\alpha}$ is a boundary, which is a contradiction. Therefore $\hocolim\widehat{F}_2$ does not have $p$-torsion in homology, completing the proof.
  \end{proof}

  \begin{corollary}
    \label{Delta torsion}
    If the mod $p$ homology of $\Delta_p(k)$ is non-trivial for some $k$, then $\Hom(\Z^m,G)_1$ has $p$-torsion in homology for each $m\ge 2$.
  \end{corollary}

  \begin{proof}
    By Theorem \ref{Delta torsion F}, if the mod $p$ homology of $\Delta_p(k)$ is non-trivial for some $k$, then $\hocolim\widehat{F}_2$ has $p$-torsion in homology. Thus by Proposition \ref{direct summand}, $\Hom(\Z^2,G)_1$ has $p$-torsion in homology too. Since $\Hom(\Z^2,G)_1$ is a retract of $\Hom(\Z^m,G)_1$ for $m\ge 2$, the proof is complete.
  \end{proof}


  \section{Computation of torsion in homology}\label{Computation of torsion in homology}

  This section computes torsion in the homology of $\Hom(\Z^m,G)_1$ for $G$ of type $A$, $D$ and exceptional type by describing the complex $\Delta_p(k)$ in terms of the extended Dynkin diagram of $G$.

  Note that every facet of $\Delta$ corresponds to a simple root or the highest root and every $i$-face is an intersection of $n-i$ facets. Then there is a one-to-one correspondence between $i$-faces of $\Delta$ and choices of $n-i$ simple roots and the highest root. Recall that the extended Dynkin diagram of $G$ is a graph whose vertices are simple roots and the highest root. We will mean by a colored extended Dynkin diagram of $G$ an extended Dynkin diagram of $G$ whose vertices are colored by black and white. Here is an example of a colored extended Dynkin diagram of $Spin(12)$.

  
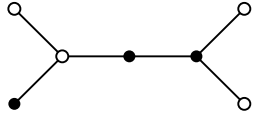
\begin{figure}[H]
    \scalebox{0.9}{
    \begin{tikzpicture}[x=0.7cm, y=0.7cm, thick]
      \draw(-1,-1)--(0,0)--(-1,1);
      \draw(0,0)--(2.8,0);
      \draw(3.8,-1)--(2.8,0)--(3.8,1);
      \fill[white](0,0) circle[radius=2.5pt];
      \draw(0,0) circle[radius=2.5pt];
      \fill(-1,-1) circle[radius=2.5pt];
      \fill[white](-1,1) circle[radius=2.5pt];
      \draw(-1,1) circle[radius=2.5pt];
      \fill(1.4,0) circle[radius=2.5pt];
      \fill(2.8,0) circle[radius=2.5pt];
      \fill[white](3.8,-1) circle[radius=2.5pt];
      \draw(3.8,-1) circle[radius=2.5pt];
      \fill[white](3.8,1) circle[radius=2.5pt];
      \draw(3.8,1) circle[radius=2.5pt];
    \end{tikzpicture}}
  \end{figure}

  Let $\mathbb{D}_i$ be the set of all colored extended Dynkin diagram with $i+1$ white vertices and $n-i$ black vertices, where the extended Dynkin diagram of $G$ has $n+1$ vertices. Then by the observation above, we get:

  \begin{lemma}
    \label{Psi_k}
    There is a bijection
    \[
      \Psi_i\colon P_i(\Delta)\xrightarrow{\cong}\mathbb{D}_i
    \]
    which sends an $i$-face $\sigma\in P_i(\Delta)$ to a colored extended Dynkin diagram such that only $n-i$ vertices corresponding to $\sigma$ are black-colored.
  \end{lemma}

  We will compute the mod $p$ homology of $\Delta_p(k)$ by specifying $\Psi_i(P_i(\Delta_p(k)))$. Let $\Gamma$ be a colored extended Dynkin diagram of $G$. For an induced subgraph $\Theta$ of $\Gamma$, let $W_\Theta$ denote the subgroup of $W$ generated by simple reflections corresponding to the vertices of $\Theta$, where we put $W_\emptyset=1$. By definition, we have:

  \begin{lemma}
    \label{W_Gamma}
    \begin{enumerate}
      \item If a colored extended Dynkin diagram $\Gamma$ is the disjoint union of induced subgraphs $\Gamma_1,\ldots,\Gamma_k$ after removing all white vertices, then
      \[
        W_\Gamma=W_{\Gamma_1}\times\cdots\times W_{\Gamma_k}.
      \]

      \item For an $i$-face $\sigma$ of $\Delta$, there is an isomorphism
      \[
        W(\sigma)\cong W_{\Psi_i(\sigma)}.
      \]
    \end{enumerate}
  \end{lemma}

  Let $v_1,\ldots,v_i$ be vertices of the extended Dynkin diagram. We denote by $v_{i_1}\cdots v_{i_k}$ an $(k-1)$-face $\sigma$ such that the white vertices of the extended Dynkin diagram $\Psi_{i-1}(\sigma)$ are $v_{i_1},\ldots,v_{i_k}$. For example, as for $G=SU(3)$, 13 corresponds the following colored extended Dynkin diagram.

  
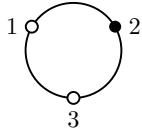
\begin{figure}[H]
    \scalebox{0.9}{
    \begin{tikzpicture}[x=0.7cm, y=0.7cm, thick]
      \draw(0,0) circle[radius=1];
      \fill(0.87,0.5) circle[radius=2.5pt];
      \fill[white](-0.87,0.5) circle[radius=2.5pt];
      \draw(-0.87,0.5) circle[radius=2.5pt];
      \fill[white](0,-1) circle[radius=2.5pt];
      \draw(0,-1) circle[radius=2.5pt];
      \draw(0.97,0.5) node[right]{2};
      \draw(-0.97,0.5) node[left]{1};
      \draw(0,-1.1) node[below]{3};
    \end{tikzpicture}}
  \end{figure}


  \subsection{Type $A$}

  Throughout this subsection, let $G=SU(n+1)$. Recall that the extended Dynkin diagram of $SU(n+1)$ is the cycle graph with $n+1$ vertices, denoted by $C_{n+1}$.

  \begin{example}
    \label{SU(3)}
    The following figure shows all colored extended Dynkin diagrams of $SU(3)$, where $\Delta$ is a 2-simplex. The left three graphs correspond to vertices, the middle three graphs correspond to edges, and the most right graph corresponds to the 2-cell.

    
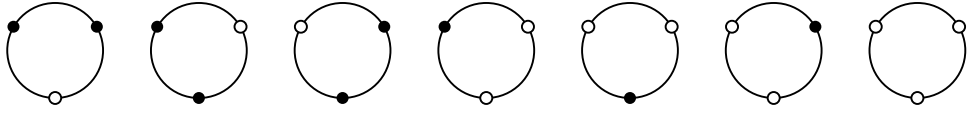
\begin{figure}[H]
      \scalebox{0.9}{
      \begin{tikzpicture}[x=0.7cm, y=0.7cm, thick]
        \draw(0,0) circle[radius=1];
        \draw(3,0) circle[radius=1];
        \draw(6,0) circle[radius=1];
        \fill(0.87,0.5) circle[radius=2.5pt];
        \fill(-0.87,0.5) circle[radius=2.5pt];
        \fill[white](0,-1) circle[radius=2.5pt];
        \draw(0,-1) circle[radius=2.5pt];
        \fill[white](3.87,0.5) circle[radius=2.5pt];
        \fill(2.13,0.5) circle[radius=2.5pt];
        \fill(3,-1) circle[radius=2.5pt];
        \draw(3.87,0.5) circle[radius=2.5pt];
        \fill(6.87,0.5) circle[radius=2.5pt];
        \fill[white](5.13,0.5) circle[radius=2.5pt];
        \fill(6,-1) circle[radius=2.5pt];
        \draw(5.13,0.5) circle[radius=2.5pt];
        \draw(9,0) circle[radius=1];
        \draw(12,0) circle[radius=1];
        \draw(15,0) circle[radius=1];
        \fill[white](9.87,0.5) circle[radius=2.5pt];
        \fill(8.13,0.5) circle[radius=2.5pt];
        \fill[white](9,-1) circle[radius=2.5pt];
        \draw(9,-1) circle[radius=2.5pt];
        \draw(9.87,0.5) circle[radius=2.5pt];
        \fill[white](12.87,0.5) circle[radius=2.5pt];
        \fill[white](11.13,0.5) circle[radius=2.5pt];
        \fill(12,-1) circle[radius=2.5pt];
        \draw(12.87,0.5) circle[radius=2.5pt];
        \draw(11.13,0.5) circle[radius=2.5pt];
        \fill(15.87,0.5) circle[radius=2.5pt];
        \fill[white](14.13,0.5) circle[radius=2.5pt];
        \fill[white](15,-1) circle[radius=2.5pt];
        \draw(14.13,0.5) circle[radius=2.5pt];
        \draw(15,-1) circle[radius=2.5pt];
        \draw(18,0) circle[radius=1];
        \fill[white](18.87,0.5) circle[radius=2.5pt];
        \draw(18.87,0.5) circle[radius=2.5pt];
        \fill[white](17.13,0.5) circle[radius=2.5pt];
        \draw(17.13,0.5) circle[radius=2.5pt];
        \fill[white](18,-1) circle[radius=2.5pt];
        \draw(18,-1) circle[radius=2.5pt];
      \end{tikzpicture}}
    \end{figure}
  \end{example}

  Now we describe faces of $\Delta_p(0)$. Let $C(i)$ be the following graph with $p^i-1$ black vertices and one gray vertex.

  
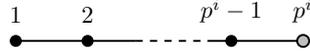
\begin{figure}[H]
    \scalebox{0.9}{
    \begin{tikzpicture}[x=0.7cm, y=0.7cm, thick]
      \draw(0,0.2) node[above]{$1$};
      \draw(1.5,0.2) node[above]{$2$};
      \draw(4.5,0.2) node[above]{$p^i-1$};
      \draw(6,0.2) node[above]{$p^i$};
      \draw(0,0)--(2.5,0);
      \draw(4,0)--(6,0);
      \draw[dashed](2.5,0)--(4,0);
      \fill(0,0) circle[radius=2.5pt];
      \fill(1.5,0) circle[radius=2.5pt];
      \fill(4.5,0) circle[radius=2.5pt];
      \fill[lightgray](6,0) circle[radius=2.5pt];
      \draw(6,0) circle[radius=2.5pt];
    \end{tikzpicture}}
  \end{figure}

  \begin{proposition}
    \label{cycle graph}
    Let $n>1$ and $n+1=\sum_{j=0}^la_jp^j$ be the $p$-adic expansion, where $0\le a_j<p$ for each $j$. A colored extended Dynkin diagram of $SU(n+1)$ is in $\Psi_i(P_i(\Delta_p(0)))$ if and only if it is obtained by gluing $a_j$ copies of $C(j)$ for $j=0,\ldots,l$ such that $i+1$ gray vertices are replaced by white-colored vertices and other $\sum_{j=0}^la_j-i$ gray vertices are replaced by black-colored by vertices.
  \end{proposition}

  \begin{proof}
    First, we prove the if part. Let $\Gamma$ be the colored extended Dynkin diagram specified in the statement, and let $\Gamma'$ be a colored extended Dynkin diagram which is constructed by repainting the originally gray vertices of $\Gamma$ by white. Then by Lemma \ref{W_Gamma},
  		\[
  		  |W_{\Gamma'}|=\prod_{j=0}^l(p^j!)^{a_j},
  		\]
  		and so we get
  		\[
  		  \frac{|W|}{|W_{\Gamma'}|}=\prod_{j=1}^{l}\prod_{j'=1}^{a_j-1}\binom{n+1-j'p^j\sum_{k>j}a_kp^k}{p^j}.
  		\]
  	By Lucas's theorem, $\binom{n+1-j'p^j\sum_{k>j}a_kp^k}{p^j}$ is prime to $p$ for each $j,j'$, implying that $\frac{|W|}{|W_{\Gamma'}|}$ is prime to $p$. Since $|W_{\Gamma'}|$ divides $|W_\Gamma|$, $\frac{|W|}{|W_{\Gamma}|}$ is prime to $p$, so that we obtain $\Gamma \in \Psi_i(P_i(\Delta_p(0)))$.

    Next, we prove the only if part. The case $n=2$ can be easily deduced from Example \ref{SU(3)}. Suppose we have the one-to-one correspondence in the statement for $SU(k+1)$ with $k\le n-1$. Since $W(v)=W$ for each vertex $v$ of $\Delta$, vertices of $\Delta$ are vertices of $\Delta_p(0)$. Then by Lemma \ref{Psi_k}, there is a one-to-one correspondence between $\Psi_0(P_0(\Delta_p(0)))$ and vertices of $\Delta_p(0)$.

    Let $\Gamma\in\Psi_i(P_i(\Delta_p(0)))$ for $i>0$. Then $\Gamma$ includes the following subgraph $\Theta$ with $n'$ black vertices and two white vertices for $0\le n'\le n-2$.

    
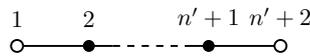
\begin{figure}[H]
      \scalebox{0.9}{
      \begin{tikzpicture}[x=0.7cm, y=0.7cm, thick]
        \draw(0,0.2) node[above]{$1$};
        \draw(1.5,0.2) node[above]{$2$};
        \draw(4,0.2) node[above]{$n'+1$};
        \draw(5.5,0.2) node[above]{$n'+2$};
        \draw(0,0)--(2,0);
        \draw(3.5,0)--(5.5,0);
        \draw[dashed](2,0)--(3.5,0);
        \fill[white](0,0) circle[radius=2.5pt];
        \draw(0,0) circle[radius=2.5pt];
        \fill(1.5,0) circle[radius=2.5pt];
        \fill(4,0) circle[radius=2.5pt];
        \fill[white](5.5,0) circle[radius=2.5pt];
        \draw(5.5,0) circle[radius=2.5pt];
      \end{tikzpicture}}
    \end{figure}

    \noindent Let $n'+1=\sum_{j=0}^la_j'p^j$ be the $p$-adic expansion. By Lemma \ref{W_Gamma},
    \[
      |W_\Gamma|=|W_\Theta||W_{\Gamma-\Theta}|
    \]
    such that $|W_\Theta|=(n'+1)!$ and $|W_{\Gamma-\Theta}|$ divides $(n-n')!$. Then since $|W|/|W_\Gamma|$ is prime to $p$,
    \[
      \frac{|W|}{(n'+1)!(n-n')!}=\frac{(n+1)!}{(n'+1)!(n-n')!}=\binom{n+1}{n'+1}
    \]
    is also prime to $p$. Thus by Lucas's theorem, we obtain
    \[
      a_j'\le a_j
    \]
    for each $j$. Let $\overline{\Theta}$ be a cycle graph with $n-n'+1$ vertices which is obtained from $\Gamma$ by contracting $\Theta$ to a single white vertex. Then by the induction hypothesis, $\overline{\Theta}$ belongs to $\Psi_{i-1}(\Delta_p(0))$ for $G=SU(n-n')$, and so $\overline{\Theta}$ is obtained by gluing $C(j)$. By definition, the graph $\Theta$ is also obtained by gluing $C(j)$. Therefore $\Gamma$ itself is obtained by gluing $C(j)$, completing the proof.
  \end{proof}

  Now we compute torsion in the homology of $\Hom(\Z^m,SU(n+1))_1$.

  \begin{theorem}
    \label{torsion SU}
    The homology of $\Hom(\Z^m,SU(n+1))_1$ for $m\ge 2$ has $p$-torsion in homology if and only if $p\le n+1$.
  \end{theorem}

  \begin{proof}
    As mentioned in Section \ref{Introduction}, $\Hom(\Z^m,SU(n+1))_1$ has no $p$-torsion in homology for $p>n+1$. So we assume $p\le n+1$ and prove that $\Hom(\Z^m,SU(n+1))_1$ has $p$-torsion in homology. By Corollary \ref{Delta torsion}, it suffices to show that the mod $p$ homology of $\Delta_p(0)$ for $SU(n+1)$ is non-trivial. To this end, we aim to prove $\chi(\Delta_p(0))\ne 1$, where $\chi(K)$ denotes the Euler characteristic of a simplicial complex $K$. By Lemma \ref{Psi_k} and Proposition \ref{cycle graph}, rotations of a cycle graph induce the action of a cyclic group $\Z/(n+1)$ on $\Psi_i(P_i(\Delta_p(0)))$. Let $n+1=\sum_{j=0}^la_jp^j$ be the $p$-adic expansion, where $a_l\ne 0$. Then every element of the stabilizer of a face $\sigma\in P_i(\Delta_p(0))$ permutes $C(l)$-parts of $\Psi_i(\sigma)$ because $C(l)$ is not obtained by gluing $a_i$ copies of $C(i)$ for $i<l$. Thus the order of the stabilizer of $\Psi_i(\sigma)$ is at most $(n+1,a_l)$, implying $|P_i(\Delta_p(0))|$ is divisible by
    \[
      1<\frac{n+1}{(n+1,a_l)}<n+1
    \]
    for each $i$. Therefore since $a_l<n+1$ and $\chi(\Delta_p(0))=\sum_{i\ge 0}(-1)^i|P_i(\Delta_p(0))|$, we obtain that $\chi(\Delta_p(0))$ is divisible by $\frac{n+1}{(n+1,a_l)}$, implying $\chi(\Delta_p(0))\ne 1$.
  \end{proof}


  \subsection{Type $D$}

  Throughout this subsection, let $G=Spin(2n)$, and let $r$ be the integer such that $|W|=p^rq$ for $(p,q)=1$. We aim to prove the non-triviality of the homology of $\Delta_p(r-1)$. Fix a vertex $v$ of the extended Dynkin diagram of $SU(n+1)$. Let $f_i(n,p)$ denote the number of colored extended Dynkin diagrams $\Gamma$ of $SU(n+1)$ such that $\Gamma\in\mathbb{D}_i$, $v$ is white-colored and $|W_\Gamma|$ is prime to $p$. Let
  \[
    \chi(n,p)=\sum_{i=0}^n(-1)^if_i(n,p).
  \]

  \begin{example}
    All colored extended Dynkin diagrams $\Gamma$ of $SU(4)$ satisfying that a fixed vertex $v$ is white-colored and $|W_\Gamma|$ is prime to $p=3$ are as below. Then $f_0=0,\,f_1=1,\,f_2=3,\,f_3=1$, implying $\chi(3,3)=1$.

    
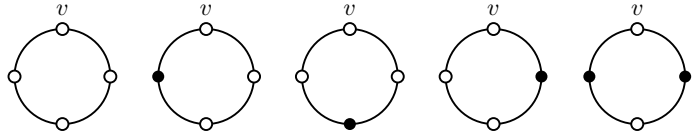
\begin{figure}[H]
      \scalebox{0.9}{
      \begin{tikzpicture}[x=0.7cm, y=0.7cm, thick]
        \draw(0,0) circle[radius=1];
        \draw(3,0) circle[radius=1];
        \draw(6,0) circle[radius=1];
        \draw(9,0) circle[radius=1];
        \draw(12,0) circle[radius=1];
        \draw(0,1.1) node[above]{$v$};
        \draw(3,1.1) node[above]{$v$};
        \draw(6,1.1) node[above]{$v$};
        \draw(9,1.1) node[above]{$v$};
        \draw(12,1.1) node[above]{$v$};
        \fill[white](0,1) circle[radius=2.5pt];
        \draw(0,1) circle[radius=2.5pt];
        \fill[white](0,-1) circle[radius=2.5pt];
        \draw(0,-1) circle[radius=2.5pt];
        \fill[white](1,0) circle[radius=2.5pt];
        \draw(1,0) circle[radius=2.5pt];
        \fill[white](-1,0) circle[radius=2.5pt];
        \draw(-1,0) circle[radius=2.5pt];
        \fill[white](3,1) circle[radius=2.5pt];
        \draw(3,1) circle[radius=2.5pt];
        \fill[white](3,-1) circle[radius=2.5pt];
        \draw(3,-1) circle[radius=2.5pt];
        \fill[white](4,0) circle[radius=2.5pt];
        \draw(4,0) circle[radius=2.5pt];
        \fill(2,0) circle[radius=2.5pt];
        \fill[white](6,1) circle[radius=2.5pt];
        \draw(6,1) circle[radius=2.5pt];
        \fill(6,-1) circle[radius=2.5pt];
        \fill[white](7,0) circle[radius=2.5pt];
        \draw(7,0) circle[radius=2.5pt];
        \fill[white](5,0) circle[radius=2.5pt];
        \draw(5,0) circle[radius=2.5pt];
        \fill[white](9,1) circle[radius=2.5pt];
        \draw(9,1) circle[radius=2.5pt];
        \fill[white](9,-1) circle[radius=2.5pt];
        \draw(9,-1) circle[radius=2.5pt];
        \fill(10,0) circle[radius=2.5pt];
        \fill[white](8,0) circle[radius=2.5pt];
        \draw(8,0) circle[radius=2.5pt];
        \fill[white](12,1) circle[radius=2.5pt];
        \draw(12,1) circle[radius=2.5pt];
        \fill[white](12,-1) circle[radius=2.5pt];
        \draw(12,-1) circle[radius=2.5pt];
        \fill(13,0) circle[radius=2.5pt];
        \fill(11,0) circle[radius=2.5pt];
      \end{tikzpicture}}
    \end{figure}
  \end{example}

  We compute $\chi(n,p)$ for general $n$ and $p$.

  \begin{lemma}
    \label{chi(n,p)}
    We have
    \[
      \chi(n,p)=
      \begin{cases}
        -1&n\equiv-1\mod p\\
        1&n\equiv 0\mod p\\
        0&\text{otherwise.}
      \end{cases}
    \]
  \end{lemma}

  \begin{proof}
    It is easy to see the statement holds for $n\le p-1$. Let $\Gamma$ be a colored extended Dynkin diagram of $SU(n+1)$ such that $v$ is white-colored and $|W_\Gamma|$ is prime to $p$. Then for $n\ge p$, it follows from Lemma \ref{W_Gamma} that $\Gamma$ is given as below, where $0\le k\le p-2$.

    
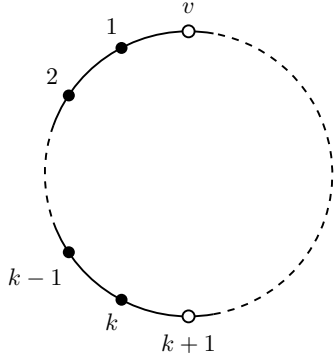
\begin{figure}[H]
      \scalebox{0.9}{
      \begin{tikzpicture}[x=0.7cm, y=0.7cm, thick]
        \draw([shift={(0,0)}]80:3) arc [radius=3, start angle=80, end angle=160];
        \draw[dashed]([shift={(0,0)}]160:3) arc [radius=3, start angle=160, end angle=200];
        \draw([shift={(0,0)}]200:3) arc [radius=3, start angle=200, end angle=280];
        \draw[dashed]([shift={(0,0)}]280:3) arc [radius=3, start angle=280, end angle=440];
        \fill[white](0,3) circle[radius=2.5pt];
        \draw(0,3) circle[radius=2.5pt];
        \fill[white](0,-3) circle[radius=2.5pt];
        \draw(0,-3) circle[radius=2.5pt];
        \fill(-1.4,2.65) circle[radius=2.5pt];
        \fill(-2.5,1.65) circle[radius=2.5pt];
        \fill(-1.4,-2.65) circle[radius=2.5pt];
        \fill(-2.5,-1.65) circle[radius=2.5pt];
        \draw(0,3.2) node[above]{$v$};
        \draw(0,-3.2) node[below]{$k+1$};
        \draw(-1.6,2.75) node[above]{$1$};
        \draw(-2.85,1.7) node[above]{$2$};
        \draw(-3.2,-1.8) node[below]{$k-1$};
        \draw(-1.6,-2.75) node[below]{$k$};
      \end{tikzpicture}}
    \end{figure}
    \noindent Thus we get
    \[
      \chi(n,p)=\sum_{i=1}^{p-2}-\chi(n-i,p)
    \]
    and so the proof is done by induction on $n$.
  \end{proof}

  We name the vertices of the extended Dynkin diagram of $Spin(2n)$ as follows.

  
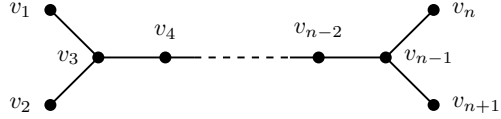
\begin{figure}[H]
    \scalebox{0.9}{
    \begin{tikzpicture}[x=0.7cm, y=0.7cm, thick]
      \draw(-1,-1)--(0,0)--(-1,1);
      \draw(0,0)--(2,0);
      \draw(4,0)--(6,0);
      \draw[dashed](2,0)--(4,0);
      \draw(7,-1)--(6,0)--(7,1);
      \fill(0,0) circle[radius=2.5pt];
      \fill(-1,-1) circle[radius=2.5pt];
      \fill(-1,1) circle[radius=2.5pt];
      \fill(1.4,0) circle[radius=2.5pt];
      \fill(4.6,0) circle[radius=2.5pt];
      \fill(6,0) circle[radius=2.5pt];
      \fill(7,-1) circle[radius=2.5pt];
      \fill(7,1) circle[radius=2.5pt];
      \draw(-1.2,1) node[left]{$v_1$};
      \draw(-1.2,-1) node[left]{$v_2$};
      \draw(-0.2,0) node[left]{$v_3$};
      \draw(1.4,0.2) node[above]{$v_4$};
      \draw(4.6,0.2) node[above]{$v_{n-2}$};
      \draw(6.2,0) node[right]{$v_{n-1}$};
      \draw(7.2,1) node[right]{$v_n$};
      \draw(7.2,-1) node[right]{$v_{n+1}$};
    \end{tikzpicture}}
  \end{figure}

  \noindent Let $\Gamma_1$ and $\Gamma_2$ be colored extended Dynkin diagrams of $Spin(2n)$. Suppose all but one vertices of $\Gamma_1$ and $\Gamma_2$ have the same colors. If $\Gamma_1$ and $\Gamma_2$ correspond to faces of $\Delta_p(k)$, then we can cancel out these faces in the computation of $\chi(\Delta_p(k))$. Thus we specify the case $\Gamma_1$ is a face of $\Delta_p(k)$ but $\Gamma_2$ is not.

  \begin{lemma}
    \label{one vertex difference}
    Let $\Gamma_1$ and $\Gamma_2$ be colored extended Dynkin diagrams of $Spin(2n)$ such that all vertices but the vertex $v_1$ have the same color. If $\Gamma_1$ corresponds to a face of $\Delta_p(r-1)$ and $\Gamma_2$ is not, then $\Gamma_2$ is of the following form.

    \begin{figure}[H]
      \scalebox{0.9}{
      \begin{tikzpicture}[x=0.7cm, y=0.7cm, thick]
        \draw(-1,-1)--(0,0)--(-1,1);
        \draw(0,0)--(2,0);
        \draw(4,0)--(6.6,0);
        \draw[dashed](2,0)--(4,0);
        \draw[dashed](6.6,0)--(8.6,0);
        \fill(0,0) circle[radius=2.5pt];
        \fill[white](-1,1) circle[radius=2.5pt];
        \draw(-1,1) circle[radius=2.5pt];
        \fill[white](-1,-1) circle[radius=2.5pt];
        \draw(-1,-1) circle[radius=2.5pt];
        \fill(1.4,0) circle[radius=2.5pt];
        \fill(4.6,0) circle[radius=2.5pt];
        \fill[white](6.2,0) circle[radius=2.5pt];
        \draw(6.2,0) circle[radius=2.5pt];
        \draw(-1.2,-1) node[left]{$v_2$};
        \draw(-1.2,1) node[left]{$v_1$};
        \draw(-0.2,0) node[left]{$v_3$};
        \draw(1.4,0.2) node[above]{$v_4$};
        \draw(4.6,0.2) node[above]{$v_{p}$};
        \draw(6.2,0.2) node[above]{$v_{p+1}$};
      \end{tikzpicture}}
    \end{figure}
  \end{lemma}

  \begin{proof}
    The statement follows from Lemma \ref{W_Gamma}.
  \end{proof}

  \begin{theorem}
    \label{torsion Spin}
    If $p\le n$ and $n\equiv 0,1\mod p$, then for $m\ge 2$, $\Hom(\Z^m,Spin(2n))_1$ has $p$-torsion in homology.
  \end{theorem}

  \begin{proof}
    As in the proof of Theorem \ref{torsion SU}, it suffices to show $\chi(\Delta_p(r-1))\ne 1$ for $p\le n$ and $n\equiv 0,1\mod p$. Let
    \[
      \widetilde{\chi}(n,p)=\sum_{i=0}^n(-1)^i(|P_i(\Delta)|-|P_i(\Delta_p(r-1))|).
    \]
    Then since $\chi(\Delta)=1$, we have $\chi(\Delta_p(r-1))=1-\widetilde{\chi}(n,p)$, and so $\chi(\Delta_p(r-1))\ne 1$ if and only if $\widetilde{\chi}(n,p)\ne 0$.

    First we consider the $n\ge 2p+2$ case. Note that Lemma \ref{one vertex difference} holds if we replace $v_1$ with $v_n$. Then since $n\ge 2p+2$, we only need to count colored extended Dynkin diagrams of $Spin(2n)$ such that vertices $v_1,v_2,\ldots,v_{p+1},v_{n-p+1},v_{n-p+2},\ldots,v_{n+1}$ are colored as in Lemma \ref{one vertex difference}. If we delete $v_1,v_2,\ldots,v_{p+1},v_{n-p+1},v_{n-p+2},\ldots,v_{n+1}$ and add a white vertex $v$ together with edges $vv_{p+2}$ and $vv_{n-p}$, then we get a colored extended Dynkin diagram of $SU(n-2p)$. Through this operation, there is a one-to-one correspondence between colored extended Dynkin diagrams of $Spin(2n)$ whose left and right ends are as in Lemma \ref{one vertex difference} and colored extended Dynkin diagrams of $SU(n-2p)$ such that a fixed vertex $v$ is white-colored. Then we get
    \[
      \widetilde{\chi}(n,p)=-\chi(n-2p-1,p).
    \]
    Therefore, for $n\ge 2p+2$, the proof is complete by Lemma \ref{chi(n,p)}.

    Next we consider the $p\le n<2p+2$ case. We only need to count colored extended Dynkin diagrams such that vertices $v_1,v_2,\ldots,v_{p+1}$ are colored as in Lemma \ref{one vertex difference}. Except for $n=p,p+1,2p,2p+1$, vertices $v_{p+2},v_{p+3},\ldots,v_{n+1}$ can have arbitrarily color, and so $\widetilde{\chi}(n,p)=0$. For $n=p$, we only need to count the graphs whose vertices $v_1,\ldots,v_{p-1}$ are colored as in Lemma \ref{one vertex difference}. Then there are only three graphs to be counted as follows, where $v_p$ and $v_{p+1}$ are either black or white such that they cannot be white at the same time.

    
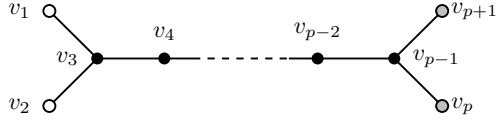
\begin{figure}[H]
      \scalebox{0.9}{
      \begin{tikzpicture}[x=0.7cm, y=0.7cm, thick]
        \draw(-1,-1)--(0,0)--(-1,1);
        \draw(0,0)--(2,0);
        \draw(4,0)--(6.2,0);
        \draw(7.2,-1)--(6.2,0)--(7.2,1);
        \draw[dashed](2,0)--(4,0);
        \fill(0,0) circle[radius=2.5pt];
        \fill[white](-1,1) circle[radius=2.5pt];
        \draw(-1,1) circle[radius=2.5pt];
        \fill[white](-1,-1) circle[radius=2.5pt];
        \draw(-1,-1) circle[radius=2.5pt];
        \fill(1.4,0) circle[radius=2.5pt];
        \fill(4.6,0) circle[radius=2.5pt];
        \fill(6.2,0) circle[radius=2.5pt];
        \fill[lightgray](7.2,1) circle[radius=2.5pt];
        \draw(7.2,1) circle[radius=2.5pt];
        \fill[lightgray](7.2,-1) circle[radius=2.5pt];
        \draw(7.2,-1) circle[radius=2.5pt];
        \draw(-1.2,-1) node[left]{$v_2$};
        \draw(-1.2,1) node[left]{$v_1$};
        \draw(-0.2,0) node[left]{$v_3$};
        \draw(1.4,0.2) node[above]{$v_4$};
        \draw(4.6,0.2) node[above]{$v_{p-2}$};
        \draw(6.4,0) node[right]{$v_{p-1}$};
        \draw(7.2,-1) node[right]{$v_{p}$};
        \draw(7.2,1) node[right]{$v_{p+1}$};
      \end{tikzpicture}}
    \end{figure}

    \noindent In each case $n=p+1,2p,2p+1$, it follows from Lemma \ref{one vertex difference} that there is only one graph to be counted. For $n=p+1$, the end vertices $v_1,v_2,v_{p+1},v_{p+2}$ are white and the remaining vertices are black. For $n=2p$, the end vertices $v_1,v_2,v_{2p},v_{2p+1}$ and the middle vertex $v_{p+1}$ are white and the remaining vertices are black. For $n=2p+1$, the end vertices $v_1,v_2,v_{2p+1},v_{2p+2}$ and the middle vertices $v_{p+1},v_{p+2}$ are white and the remaining vertices are black. Summarizing, $\widetilde{\chi}(n,p)\ne 0$ for $n=p,p+1,2p,2p+1$. Therefore the proof is complete.
  \end{proof}


  \subsection{Exceptional groups}

  We continue to compute torsion in the homology of $\Hom(\Z^m,G)_1$ when $G$ is the exceptional Lie group. Let $G$ be exceptional, and let $p$ be a prime dividing $|W|$, where $|W|$ is given as in Table \ref{order}. Then $G$ has $p$-torsion in homology except for
  \begin{equation}
    \label{possible torsion exceptional}
    (G,p)=(G_2,3),\;(E_6,5),\;(E_7,5),\;(E_7,7),\;(E_8,7).
  \end{equation}
  See \cite[Theorem 5.11, Chapter 7]{MT}. Since $G$ is a retract of $\Hom(\Z^m,G)_1$, $\Hom(\Z^m,G)_1$ has $p$-torsion in homology except possibly for the cases \eqref{possible torsion exceptional}. On the other hand, $\Hom(\Z^m,G)_1$ has no $p$-torsion in homology when $p$ does not divide $|W|$. Then we only need to consider the cases \eqref{possible torsion exceptional}.

  Now we prove:

  \begin{theorem}
    \label{torsion exceptional}
    Let $G$ be exceptional. Then $\Hom(\Z^m,G)_1$ for $m\ge 2$ has $p$-torsion in homology if and only if $p$ divides $|W|$, except possibly for $(G,p)=(E_7,5),(E_7,7),(E_8,7)$.
  \end{theorem}

  \begin{proof}
    As observed above, it follows from Corollary \ref{Delta torsion} that we only need to show that $\Delta_p(0)$ has non-trivial mod $p$ homology for $(G,p)$ in \eqref{possible torsion exceptional} for $(G,p)=(G_2,3),(E_6,5)$.

    Let $G=G_2$. Then the extended Dynkin diagram is given as below. So by Lemma \ref{W_Gamma}, $\Delta_3(0)$ consists only of two vertices 1 and 3, implying it has non-trivial mod $p$ homology.

    \begin{figure}[H]
      \scalebox{0.9}{
      \begin{tikzpicture}[x=0.7cm, y=0.7cm, thick]
        \draw(0,0)--(3,0);
        \draw(0,0.1)--(1.5,0.1);
        \draw(0,-0.1)--(1.5,-0.1);
        \draw(0.75,0)--(0.9,0.2);
        \draw(0.75,0)--(0.9,-0.2);
        \fill(0,0) circle[radius=2.5pt];
        \fill(1.5,0) circle[radius=2.5pt];
        \fill(3,0) circle[radius=2.5pt];
        \draw(0,0.2) node[above]{$1$};
        \draw(1.5,0.2) node[above]{$2$};
        \draw(3,0.2) node[above]{$3$};
      \end{tikzpicture}}
    \end{figure}

    Let $G=E_6$. Then the extended Dynkin diagram is given as below.

    \begin{figure}[H]
      \scalebox{0.9}{
      \begin{tikzpicture}[x=0.7cm, y=0.7cm, thick]
        \draw(0,0)--(6,0);
        \draw(3,0)--(3,3);
        \fill(0,0) circle[radius=2.5pt];
        \fill(1.5,0) circle[radius=2.5pt];
        \fill(3,0) circle[radius=2.5pt];
        \fill(4.5,0) circle[radius=2.5pt];
        \fill(6,0) circle[radius=2.5pt];
        \fill(3,1.5) circle[radius=2.5pt];
        \fill(3,3) circle[radius=2.5pt];
        \draw(3,-0.2) node[below]{$v$};
      \end{tikzpicture}}
    \end{figure}

    \noindent It has symmetry of rotation around the vertex $v$. By Lemma \ref{W_Gamma}, we can see that if a colored extended Dynkin diagram is symmetric with respect to the rotation around $v$, then its corresponding face of $\Delta$ does not lie in $\Delta_5(0)$. Then the number of $i$-faces of $\Delta_5(0)$ is divisible by 3 for each $i$, implying $\chi(\Delta_5(0))$ is divisible by 3. Thus $\Delta_5(0)$ has non-trivial mod 5 homology.
  \end{proof}

  By Theorems \ref{torsion SU}, \ref{torsion Spin} and \ref{torsion exceptional}, we dare to pose:

  \begin{conjecture}
    \label{conjecture}
    The homology of $\Hom(\Z^m,G)_1$ for $m\ge 2$ has $p$-torsion if and only if $p$ divides the order of $W$.
  \end{conjecture}


  \section{Negative results}\label{Negative results}

  Although we have pose Conjecture \ref{conjecture}, the complex $\Delta_p(k)$ does not work in the remaining cases, unfortunately. To be fair, we prove this, but we have to notice that those negative results do not imply the non-existence of torsion in homology because the non-triviality of the homology of $\Delta_p(k)$ is only a sufficient condition for the existence of $p$-torsion in the homology of $\Hom(\Z^m,G)_1$ for $m\ge 2$.

  First, we consider the type $D$ by examining $(G,p)=(Spin(10),3)$ which is not included in Theorem \ref{torsion Spin}. Since $|W|=2^4\cdot 5!$ for $G=Spin(10)$, we only need to consider $\Delta_3(0)$.

  \begin{proposition}
    \label{D fail}
    The complex $\Delta_3(0)$ of $Spin(10)$ is contractible.
  \end{proposition}

  \begin{proof}
    We prove the statement by applying discrete Morse theory. We refer to \cite{K} for materials of discrete Morse theory. We name the vertices of the extended Dynkin diagram of $Spin(10)$ as follows.

    
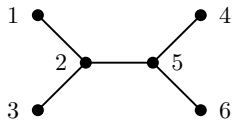
\begin{figure}[H]
      \scalebox{0.9}{
      \begin{tikzpicture}[x=0.7cm, y=0.7cm, thick]
        \draw(-1,-1)--(0,0)--(-1,1);
        \draw(0,0)--(1.4,0);
        \draw(2.4,-1)--(1.4,0)--(2.4,1);
        \fill(0,0) circle[radius=2.5pt];
        \fill(-1,-1) circle[radius=2.5pt];
        \fill(-1,1) circle[radius=2.5pt];
        \fill(1.4,0) circle[radius=2.5pt];
        \fill(2.4,1) circle[radius=2.5pt];
        \fill(2.4,-1) circle[radius=2.5pt];
        \draw(-1.2,1) node[left]{$1$};
        \draw(-0.2,0) node[left]{$2$};
        \draw(-1.2,-1) node[left]{$3$};
        \draw(2.6,1) node[right]{$4$};
        \draw(1.6,0) node[right]{$5$};
        \draw(2.6,-1) node[right]{$6$};
      \end{tikzpicture}}
    \end{figure}

    \noindent By Lemma \ref{W_Gamma}, it is straightforward to see that facets of $\Delta_3(0)$ are
    \[
      1234,\quad 1236,\quad 1346,\quad 1456,\quad 3456.
    \]
    Then we have the following acyclic partial matching.
    \begin{center}
      \begin{tabular}{lllllll}
        (1234,234)&(1236,236)&(1346,134)&(1456,145)&(3456,345)&(123,12)&(124,24)\\
        (126,26)&(136,13)&(146,14)&(156,15)&(346,34)&(356,35)&(456,45)\\
        (16,1)&(23,2)&(36,3)&(46,4)&(56,5)
      \end{tabular}
    \end{center}
    Since all faces of $\Delta_3(0)$ but the vertex 6 appear in the acyclic partial matching above, it follows from the fundamental theorem of discrete Morse theory that $\Delta_3(0)$ collapses onto the vertex 6, implying $\Delta_3(0)$ is contractible.
  \end{proof}

  Next, we consider the case of type $B$ and $C$.

  \begin{proposition}
    \label{BD fail}
    Let $G$ be $Spin(2n+1)$ for $n\geq 3$ or $Sp(n)$, and let $p$ be an odd prime dividing $|W|$. Then for each $k$, $\Delta_p(k)$ is contractible.
  \end{proposition}

  \begin{proof}
    We only prove the case $G=Spin(2n+1)$ because the case $G=Sp(n)$ is quite similarly proved. The extended Dynkin diagram of $Spin(2n+1)$ is given as follows.

    
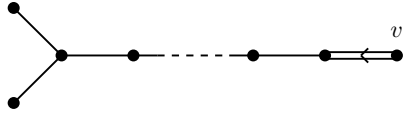
\begin{figure}[H]
      \scalebox{0.9}{
      \begin{tikzpicture}[x=0.7cm, y=0.7cm, thick]
        \draw(-1,1)--(0,0)--(-1,-1);
        \draw(0,0)--(2,0);
        \draw(3.5,0)--(5.5,0);
        \draw(5.5,0.07)--(7,0.07);
        \draw(5.5,-0.07)--(7,-0.07);
        \draw(6.25,0)--(6.4,0.15);
        \draw(6.25,0)--(6.4,-0.15);
        \draw[dashed](2,0)--(3.5,0);
        \fill(0,0) circle[radius=2.5pt];
        \fill(-1,1) circle[radius=2.5pt];
        \fill(-1,-1) circle[radius=2.5pt];
        \fill(1.5,0) circle[radius=2.5pt];
        \fill(4,0) circle[radius=2.5pt];
        \fill(5.5,0) circle[radius=2.5pt];
        \fill(7,0) circle[radius=2.5pt];
        \draw(7,0.2) node[above]{$v$};
      \end{tikzpicture}}
    \end{figure}

    \noindent Let $\widehat{\Delta}_p(k)$ be the subcomplex of $\Delta_p(k)$ consisting of faces $\sigma$ such that in the corresponding colored extended Dynkin diagram, the vertex $v$ is black. Let $\Gamma=\Psi_i(\sigma)$ for an $i$-face $\sigma$ of $\widehat{\Delta}_p(k)$. Let $\Gamma'$ be a colored extended diagram whose vertices have the same color as $\Gamma$ except for the vertex $v$. So the vertex $v$ of $\Gamma'$ is white. By Lemma \ref{W_Gamma}, $\Gamma'$ corresponds to the join $v*\sigma$. Since $|W_{\Gamma}|/|W_{\Gamma'}|$ is a power of 2, the join $v*\sigma$ is an $(i+1)$-face of $\Delta_p(k)$. Thus $\Delta_p(k)$ is the join $v*\widehat{\Delta}_p(k)$, completing the proof.
  \end{proof}

  Finally, we consider the exceptional case. The only cases that are not included in Theorem \ref{torsion exceptional} are $(G,p)=(E_7,5),(E_7,7),(E_8,7)$. Since $|W(E_7)|=2^{10}\cdot 3^4\cdot 5\cdot 7$, we only need to consider $\Delta_p(0)$ for $G=E_7$ and $p=5,7$.

  \begin{proposition}
    \label{E_7 fail}
    For $G=E_7$ and $p=5,7$, $\Delta_p(0)$ is contractible.
  \end{proposition}

  \begin{proof}
    Let $G=E_7$. Then its extended Dynkin diagram is given as follows.

    
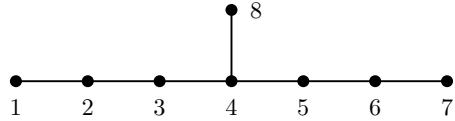
\begin{figure}[H]
      \scalebox{0.9}{
      \begin{tikzpicture}[x=0.7cm, y=0.7cm, thick]
        \draw(0,0)--(9,0);
        \draw(4.5,0)--(4.5,1.5);
        \fill(0,0) circle[radius=2.5pt];
        \fill(1.5,0) circle[radius=2.5pt];
        \fill(3,0) circle[radius=2.5pt];
        \fill(4.5,0) circle[radius=2.5pt];
        \fill(6,0) circle[radius=2.5pt];
        \fill(7.5,0) circle[radius=2.5pt];
        \fill(9,0) circle[radius=2.5pt];
        \fill(4.5,1.5) circle[radius=2.5pt];
        \draw(0,-0.2) node[below]{$1$};
        \draw(1.5,-0.2) node[below]{$2$};
        \draw(3,-0.2) node[below]{$3$};
        \draw(4.5,-0.2) node[below]{$4$};
        \draw(6,-0.2) node[below]{$5$};
        \draw(7.5,-0.2) node[below]{$6$};
        \draw(9,-0.2) node[below]{$7$};
        \draw(4.7,1.5) node[right]{$8$};
      \end{tikzpicture}}
    \end{figure}

    \noindent Then the facets of $\Delta_5(0)$ are
    \[
      1237,\quad 1238,\quad 1278,\quad 1567,\quad 1678,\quad 5678.
    \]
    So we have the following acyclic partial matching.
    \begin{center}
      \begin{tabular}{llllll}
        (1237,137)&(1238,138)&(1278,278)&(1567,157)&(1678,168)&(5678,578)\\
        (123,13)&(127,17)&(128,18)&(156,15)&(167,67)&(178,78)\\
        (237,37)&(238,38)&(567,57)&(568,58)&(678,68)&(12,2)\\
        (16,6)&(23,3)&(27,7)&(28,8)&(56,5)
      \end{tabular}
    \end{center}
    Note that all faces of $\Delta_5(0)$ but the vertex 1 appear in the acyclic partial matching above. Thus by the fundamental theorem of discrete Morse theory, $\Delta_5(0)$ collapses onto the vertex 1, implying $\Delta_5(0)$ is contractible. By Lemma \ref{W_Gamma}, the facets of $\Delta_7(0)$ are 18 and 78. Then $\Delta_7(0)$ is a path graph of length 2, implying it is contractible.
  \end{proof}

  Since $|W(E_8)|=2^{14}\cdot 3^5\cdot 5^2\cdot 7$, we only need to consider $\Delta_7(0)$ for $(G,p)=(E_8,7)$.

  \begin{proposition}
    \label{E_7 fail}
    For $G=E_8$, $\Delta_7(0)$ is contractible.
  \end{proposition}

  \begin{proof}
    The extended Dynkin diagram of $E_8$ is given as below.

    
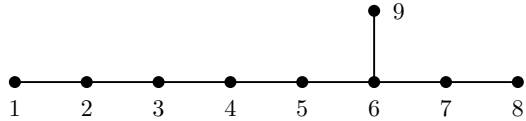
\begin{figure}[H]
      \scalebox{0.9}{
      \begin{tikzpicture}[x=0.7cm, y=0.7cm, thick]
        \draw(0,0)--(10.5,0);
        \draw(7.5,0)--(7.5,1.5);
        \fill(0,0) circle[radius=2.5pt];
        \fill(1.5,0) circle[radius=2.5pt];
        \fill(3,0) circle[radius=2.5pt];
        \fill(4.5,0) circle[radius=2.5pt];
        \fill(6,0) circle[radius=2.5pt];
        \fill(7.5,0) circle[radius=2.5pt];
        \fill(9,0) circle[radius=2.5pt];
        \fill(10.5,0) circle[radius=2.5pt];
        \fill(7.5,1.5) circle[radius=2.5pt];
        \draw(0,-0.2) node[below]{$1$};
        \draw(1.5,-0.2) node[below]{$2$};
        \draw(3,-0.2) node[below]{$3$};
        \draw(4.5,-0.2) node[below]{$4$};
        \draw(6,-0.2) node[below]{$5$};
        \draw(7.5,-0.2) node[below]{$6$};
        \draw(9,-0.2) node[below]{$7$};
        \draw(10.5,-0.2) node[below]{$8$};
        \draw(7.7,1.5) node[right]{$9$};
      \end{tikzpicture}}
    \end{figure}

    \noindent Then by Lemma \ref{W_Gamma}, $\Delta_7(0)$ is the following 2-dimensional simplicial complex.

    \begin{figure}[H]
      \scalebox{0.9}{
      \begin{tikzpicture}[x=0.7cm, y=0.7cm, thick]
        \filldraw[fill=lightgray](0,0)--(0,2)--(2,2)--(4,0)--(0,0);
        \draw(0,0)--(2,1)--(4,0);
        \draw(0,0)--(2,2)--(2,1);
        \fill(0,0) circle[radius=2.5pt];
        \fill(0,2) circle[radius=2.5pt];
        \fill(2,2) circle[radius=2.5pt];
        \fill(4,0) circle[radius=2.5pt];
        \fill(2,1) circle[radius=2.5pt];
        \draw(0,-0.2) node[below]{$9$};
        \draw(0,2.2) node[above]{$2$};
        \draw(2,2.2) node[above]{$1$};
        \draw(4,-0.2) node[below]{$7$};
        \draw(2,0.8) node[below]{$8$};
      \end{tikzpicture}}
    \end{figure}

    \noindent Thus $\Delta_7(0)$ is contractible.
  \end{proof}

\end{document}